\documentclass[a4paper,12pt]{article} 

\usepackage{color}

\usepackage{amsfonts, amsmath, amsthm, amssymb}
\usepackage[T1]{fontenc}
\usepackage[cp1250]{inputenc}
\usepackage{xcolor}
\usepackage{graphicx}
\usepackage{amssymb}
\usepackage{amsmath}
\usepackage{mathptmx}
\usepackage{helvet}
\usepackage{courier}
\usepackage{txfonts}
\usepackage{tikz} 
\usetikzlibrary{arrows}
\usepackage{type1cm}

\usepackage{verbatim}

\usepackage{graphicx}
\usepackage{epsfig,amscd,amssymb,amsxtra,amsmath,amsthm}
\usepackage{type1cm}
\usepackage[T1]{fontenc}
\usepackage{graphics}
\usepackage[mathscr]{eucal}
\usepackage[all]{xy}
\usepackage{amsmath,amscd}

\newtheorem{theorem}{Theorem}[section]
\newtheorem{prop}[theorem]{Proposition}
\newtheorem{definition}[theorem]{Definition}
\newtheorem{lemma}[theorem]{Lemma}

\newtheorem{example}[theorem]{Example}
\newtheorem{remark}[theorem]{Remark}

\newtheorem{corollary}[theorem]{Corollary}
\newtheorem{problem}[theorem]{Problem}
\newtheorem{observation}[theorem]{Observation}

\newcommand{\diam} {\mathop{\rm diam}\nolimits}

\newcommand{\N}{\mathbb{N}}
\newcommand{\Z}{\mathbb{Z}}

\newcommand{\xx}{\mathbf{x}}
\newcommand{\yy}{\mathbf{y}}
\newcommand{\zz}{\mathbf{z}}
\renewcommand{\aa}{\mathbf{a}}
\newcommand{\bb}{\mathbf{b}}
\newcommand{\cc}{\mathbf{c}}
\newcommand{\dd}{\mathbf{d}}

\newcommand{\se}{\subseteq}
\newcommand{\e}{\epsilon}

\begin{document}

\title{The Specification Property on the Lelek Fan}
\author{Goran Erceg, James Kelly, Judy Kennedy, Christopher Mouron, Van Nall}
\maketitle

\begin{abstract} 
	Recent work of Piotr Oprocha and his collaborators (see \cite{KKO1}, \cite{KKO2}, \cite{KLO}, e.g.) has provided a number of delicate examples of dynamical systems separating specification, shadowing, and periodic-point density, primarily in symbolic or totally disconnected spaces. The goal of the present paper is to demonstrate that similar - and in some cases sharper - separations occur on the Lelek fan, a smooth one-dimensional continuum. 

	Our constructions rely on Mahavier products of closed relations. By carefully choosing relations on the unit interval, we obtain Mahavier products that are homeomorphic to the Lelek fan whose associated shift maps display diverse dynamical behavior. This approach yields a unified framework for producing and analyzing examples on a familiar continuum.  
	
\end{abstract}
\-
\vphantom{}

\vphantom{}

\section{Introduction}
The specification property, introduced by Rufus Bowen (see \cite{bowen}) in 1971 for investigating  Axiom A diffeomorphisms, builds on earlier work of Stephen Smale on hyperbolic systems. It formalizes the ability of a dynamical system to trace arbitrary finite collections of orbit segments by a single orbit, provided there is sufficient spacing between orbit segments. Closely related notions include shadowing, mixing, transitivity, and density of periodic points. Although these properties coincide in many hyperbolic settings, they can be separated in more general topological dynamics.  

The tent map, defined by $T(x)=1-|1-2x|$ for $x \in [0,1],$ is a fundamental, elementary example in chaotic dynamics, exhibiting sensitive dependence on initial conditions, dense periodic points, and topological mixing, and it also has the specification property. The specification property has a serious connection to thermodynamic formalism as it allows for the application of advanced results in ergodic theory. It guarantees the uniqueness of equilibrium states (Gibbs measures) and a strong relationship between topological entropy, potentials, pressure, and ergodic theory.

We construct continuous surjections on the Lelek fan exhibiting a wide range of dynamical behaviors. Using Mahavier products of closed relations, we produce examples that separate classical dynamical properties such as transitivity, mixing, the specification property, shadowing, density of periodic points, and positive entropy. In particular, we show that on the Lelek fan one may have: (i) Devaney chaos without shadowing; (ii) transitivity and positive entropy with a unique periodic point but without specification, shadowing, or mixing; and (iii) mixing together with the specification property but without shadowing. These examples extend earlier constructions in symbolic and totally disconnected settings to a classical one-dimensional continuum. None of our examples has the shadowing property.

\section{Definitions, Notation and Background}

\begin{definition} Let $X$ be a compact metric space. We use $\mathbb{N}$ to denote the positive integers; $\mathbb{N}_0$ to denote the nonnegative integers; and $\mathbb{Z}$ to denote the integers.

\end{definition}

\begin{definition} Let $X$ be a compact metric space. We use $d$ to denote the metric on $X$, unless otherwise specified. We use $d_H$ to denote the Hausdorff metric on the space of closed subsets of $X$, unless otherwise specified. For the product space $X^{\infty}$, we use the convention that a point $\mathbf{x} \in X^{\infty}$ has coordinates $x_i, i \ge 1$, i.e., $\mathbf{x}=(x_1,x_2, \ldots)$, and we use the metric 

$$D(\mathbf{x},\mathbf{y})= \sup \left\lbrace \frac{|x_i-y_i|}{2^i}:i \ge 1 \right\rbrace.$$

\end{definition}

\begin{definition}(\emph{The standard projection} $\pi_i$.) Suppose $X$ is a space. If $n$ is a positive integer, and $\mathbf{x}=(x_1,x_2, \ldots, x_n) \in X^n$, then $\pi_i(\mathbf{x}) = x_i$ for $i \in \{1,2,\dots,n\}$. If $\mathbf{x}=(x_1,x_2, \ldots) \in X^{\infty}$, then $\pi_i(\mathbf{x}) = x_i$ for $i$ each positive integer $i$. If $m \le n$, $m$ and $n$ are integers, and if $\mathbf{x}=(x_1,x_2,\dots)$, then $\pi_{[m,n]}(\mathbf{x})=(x_m,x_{m+1},\ldots, x_n)$, when $0 <m \le n$, $m,n$ integers.
	
\end{definition}

\begin{definition} A \emph{continuum} is a nonempty compact connected metric space. A continuum is \emph{degenerate} if it consists of exactly one point. Otherwise it is \emph{nondegenerate}. A subspace of a continuum which is itself a continuum is a \emph{subcontinuum}. 
	
\end{definition}

\begin{definition} Let $X$ and $Y$ be metric spaces, and let $f:X \to Y$ be a function. We use $\Gamma(F)= \{(x,y): y=f(x) \}$ to denote the \emph{graph of the function} $f$.
	
\end{definition}

 \begin{definition}
 Let $X$ be a compact metric space and let $G \subset X \times X$ be a relation on $X$. If $G$ is closed in $X \times X$, then $G$ is a \emph{closed relation on} $X$.  
 \end{definition}

\begin{definition} Let $X$ be a set and let $G$ be a relation on $X$. Then we define $G^{-1}=\{ (y,x): (x,y) \in G \}$. Also, $G^{-1}$ is the \emph{inverse relation of the relation $G$ on} $X$.
	
\end{definition}

\begin{definition} A \emph{dynamical system} is a pair $(X,f)$ where $X$ is a compact metric space and $f:X \to X$ is continuous. We often identify the dynamical system $(X,f)$ with the map $f:X \to X$.

\end{definition}

\begin{definition} If $X$ is a compact metric space and $f : X \to X$ is continuous, we say that the point $x \in X$ is a \emph{periodic point} for $f$ if $f^ n(x) = x$ for some $n \in \mathbb{N}$ and call $n$ a \emph{period} for $x$. If $f (x) = x$, we say $x$ is a \emph{fixed point} for $f$.
	
\end{definition} 

\begin{definition}
Let $(X, f )$ be a dynamical system. We say that $(X, f )$ is \emph{transitive} if for all non-empty open sets $U$ and $V$ in $X$, there is a non-negative integer $n$ such that $f ^n(U ) \cap V \ne \emptyset$. We say that the mapping $f$ is \emph{transitive} if $(X, f )$ is transitive.
	
\end{definition}

\begin{definition} Let $(X, f )$ be a dynamical system. We say that $(X, f )$ is \emph{topologically mixing} if for all non-empty open sets $U$ and $V$ in $X$, there is a non-negative integer $n_0$ such that for each positive integer $n$,
$n \ge n_0$ implies $f^n(U ) \cap  V \ne  \emptyset$. We say that the mapping $f$ is \emph{mixing} if $(X,f)$ is mixing.
\end{definition}

 \begin{definition}
 
 Let $(X, f )$ and $(Y, g)$ be dynamical systems. If there is a homeomorphism $\phi : X \to Y$ such that
$$ \phi \circ  f = g \circ \phi,$$
then we say that $(X, f )$ and $(Y, g)$ are \emph{topological conjugates}.

 \end{definition} 
 
\begin{definition}\label{def:delta-pseudo}
	Let $(X,f)$ be a dynamical system, let $\delta>0$ and let $(x_0, x_1, x_2, x_3, \ldots)$
	be a sequence in $X$. We say that $(x_0, x_1, x_2, x_3, \ldots)$ is a \emph{$\delta$-pseudo-orbit} in $(X,f)$ if for each non-negative integer $n$,
	\[
	d(f(x_n), x_{n+1}) < \delta.
	\]
\end{definition}

\begin{definition}\label{def:shadowing}
	Let $(X,f)$ be a dynamical system. We say that $(X,f)$ has the \emph{shadowing property} if for each $\varepsilon > 0$, there exists $\delta > 0$ such that for each $\delta$-pseudo-orbit $(x_k)$ in $(X,f)$, there exists $y \in X$ such that for each non-negative integer $k$,
	\[
	d(f^k(y), x_k) < \varepsilon.
	\]
\end{definition}

\begin{observation}\label{obs:conjugate-shadowing}
	Let $(X,f)$ and $(Y,g)$ be dynamical systems such that $(X,f)$ has the shadowing property.
	If $(Y,g)$ is topologically conjugate to $(X,f)$, then $(Y,g)$ also has the shadowing property.
\end{observation}

\begin{definition}
	Let $(X,f)$ be a dynamical system, let $x\in X$ and let $k,\ell$ be non-negative integers. If $k\leq \ell$, then we say that 
	$$
	f^{[k,\ell]}(x)=\Big(f^k(x),f^{k+1}(x),f^{k+2}(x),\ldots,f^{\ell}(x)\Big)
	$$
	is the \emph{$[k,\ell]$-orbit segment of the point $x$}.  
\end{definition}

\begin{definition}
	Let $(X,f)$ be a dynamical system, let $n$ be a positive integer and for each $j\in \{1,2,3,\ldots,n\}$, let 
	\begin{enumerate}
		\item $k_j$ and $\ell_j$ be non-negative integers such that $k_j\leq \ell_j$, and
		\item $x_j\in X$.
	\end{enumerate}
	We say that the $n-tuple$
	$$
	\Bigg(f^{[k_1,\ell_1]}(x_1),f^{[k_2,\ell_2]}(x_2),f^{[k_3,\ell_3]}(x_3),\ldots ,f^{[k_n,\ell_n]}(x_n)\Bigg)
	$$
	is \emph{an $n$-specification} or just \emph{a specification in $(X,f)$}.
\end{definition}
\begin{definition}
	Let $(X,f)$ be a dynamical system, let $N$ be a positive integer and let 
	$$
	\mathcal S=\Bigg(f^{[k_1,\ell_1]}(x_1),f^{[k_2,\ell_2]}(x_2),f^{[k_3,\ell_3]}(x_3),\ldots ,f^{[k_n,\ell_n]}(x_n)\Bigg)
	$$
	be a specification in $(X,f)$. We say that $\mathcal S$ is \emph{an $N$-spaced specification}, if for each $j\in\{1,2,3,\ldots,n-1\}$,
	$$
	k_{j+1}-\ell_j\geq N.
	$$
\end{definition}
\begin{definition}
	Let $(X,f)$ be a dynamical system, let $N$ be a positive integer, let $\varepsilon>0$, let $y\in X$ and let 
	$$
	\mathcal S=\Bigg(f^{[k_1,\ell_1]}(x_1),f^{[k_2,\ell_2]}(x_2),f^{[k_3,\ell_3]}(x_3),\ldots ,f^{[k_n,\ell_n]}(x_n)\Bigg)
	$$
	be an $N$-spaced specification in $(X,f)$. We say that \emph{$\mathcal S$ is $\varepsilon$-traced in $(X,f)$ by $y$} if for each $i\in \{1,2,3,\ldots,n\}$ and for each $j\in \{k_i,k_i+1,k_i+2,\ldots,\ell_i\}$,
	$$
	d(f^j(y),f^j(x_i))\leq \varepsilon.
	$$
\end{definition}
Finally, in Definition \ref{def1}, the specification property is defined.
\begin{definition}\label{def1}
	Let $(X,f)$ be a dynamical system. We say that \emph{$(X,f)$ has the specification property} if for each $\varepsilon >0$, there is a positive integer $N$ such that for any $N$-spaced specification $\mathcal S=\left(f^{[k_1,\ell_1]}(x_1),f^{[k_2,\ell_2]}(x_2),f^{[k_3,\ell_3]}(x_3),\ldots ,f^{[k_n,\ell_n]}(x_n)\right)$ in $(X,f)$, there is $y\in X$ such that $\mathcal S$ is $\varepsilon$-traced in $(X,f)$ by $y$. If, additionally, $y$ can be chosen such that $f^{\ell_n-k_1+N}(y)=y$ we say that \emph{$(X,f)$ has the periodic specification property.}
\end{definition}

\section{Mahavier products, closed relations, and CR-specification}

The previous section presented standard definitions, and notation from continuum theory and topological dynamics. In the next section we continue with definitions, observations, notation,  but the focus shifts to those developed in our work. For clarity, we separate these from the earlier material.

\begin{definition} Let $X$ be a compact metric space and let $G$ be a closed relation on $X$. 

\begin{itemize}
\item  We define the \emph{Mahavier product} $G \star 
G$ as  $\left\{(x,y,z): (x,y) \in G, (y,z) \in G\right\}$. If $\mathbf{a}=(x,y) \in G$  and $\mathbf{b}=(y,z) \in G$, then we define $\mathbf{a} \star \mathbf{b}$ to be the point $\mathbf{a} \star \mathbf{b}=(x,y,z) \in G \star G$.

\item If $m$ is a positive integer, we define the $m$-th \emph{Mahavier product} $\star_{i=1}^m G$ to be 
$\star_{i=1}^m G = \left\{(x_1,x_2, \ldots, x_{m+1}) \in \prod_{i=1}^{m+1} X: \text{ for each } i \in \{1,2,\ldots,m\} , (x_i,x_{i+1}) \in G \right\}.$ We abbreviate $\star_{i=1}^m G$ by $X_G^m$. 

\item If $n$ is an integer less than $m,$ we define \newline $\star_{i=n}^m G = \left\{(x_n,x_{n+1}, \ldots, x_{m+1}) \in \prod_{i=n}^{m+1} X: \text{ for each } i \in \{n,n+1,\ldots,m\} , (x_i,x_{i+1}) \in G \right\}.$

\item We define the  \emph{infinite Mahavier product} $\star_{i=1}^{\infty} G$ to be \newline 
$\star_{i=1}^{\infty} G = \left\{(x_1,x_2, \ldots) \in \prod_{i=1}^{\infty} X: \text{ for each positive integer } i, (x_i,x_{i+1}) \in G \right\}.$ We abbreviate $\star_{i=1}^{\infty} G$ by $X_G^+$.

\item We define the \emph{two-sided  Mahavier product} $\star_{i=-\infty}^{\infty} G$ to be \\ $\star_{i=-\infty}^{\infty} G=\left\{(\ldots,x_{-2},x_{-1},{x_0}{ ;}x_1,x_2,\ldots )\in \prod_{i={-\infty}}^{\infty}X \ | \ \text{ for each  integer } i, (x_{i},x_{i+1})\in {G}\right\}.$   We abbreviate $\star_{i=-\infty}^{\infty} G$ by $X_G$.

\item We define the function $$\sigma_G^+: X_G^+ \to X_G^+$$ given by $$\sigma_G^+(x_1,x_2, \ldots) = (x_2,x_3, \ldots)$$ for each 
$(x_1,x_2, \ldots) \in X_G^+$. The map $\sigma_G^+$ is called the \emph{shift map on} $X_G^+$, and it is a continuous function. (To see this, observe that $X_G^+$ is invariant under the shift $\sigma$ on $\Pi_{i=1}^{\infty} X$ defined by $\sigma(x_1,x_2, \dots)=(x_2,x_3, \dots)$ for $(x_1,x_2,\dots) \in \Pi_{i=1}^{\infty} X$. Since $\sigma$ is continuous, so is $\sigma|X_G^+ = \sigma_G^+$.) 

\item The function  $\sigma_G : {X_G} \rightarrow {X_G}$, 
defined by 
$$
\sigma_G (\ldots,x_{-2},x_{-1},{x_0};x_1,x_2,\ldots )=(\ldots,x_{-2},x_{-1},{x_0},x_1;x_2,\ldots )
$$
for each $(\ldots,x_{-2},x_{-1},{x_0};x_1,x_2,\ldots )\in {X_F}$, 
is called \emph{   the shift map on ${X_G}$}.    
\end{itemize}

\end{definition}

\begin{definition}
	Let $X$ be a {non-empty} compact metric space and let $F\subseteq X\times X$ be a closed relation on $X$. We say $(X,F)$ is \emph{a CR-dynamical system}.
\end{definition} 

\begin{definition}\label{luhca2}
	Let $(X,F)$ be a CR-dynamical system.  The dynamical system 
	\begin{enumerate}
		\item $(X_F^{+},\sigma_F^+)$ is called \emph{a Mahavier dynamical system}.
		\item $(X_F,\sigma_F)$ is called \emph{a two-sided Mahavier dynamical system}.
	\end{enumerate}
\end{definition}

 \begin{definition}
 	Let $\left(X,F\right)$ be a CR-dynamical system, let $\xx = (x (1),x (2),\ldots) \in X_F^+$ and let $k\leq\ell$ be positive integers.  We say that $\pi_{[k,\ell]}(\xx)=(x (k),x (k+1),\ldots,x (\ell))\in X^{\ell-k+1}$ is a $\left[k,\ell\right]$\emph{-orbit segment} of the point $\xx$. 	
 \end{definition}
 \begin{definition}
 Let $(X,F)$ be a CR-dynamical system, let $n$ be a positive integer, and for each $j \in \{1,2,
 \ldots,n\}$, let 
 \begin{enumerate}
 	\item $k_j$ and $l_j$ be positive integers such that $k_j \le l_j$, 
 	\item $\mathbf{x}_j \in X_F^+$, and 
 	\item $\pi_{[k_j,l_j]}(\mathbf{x}_j)$ be a $[k_j,l_j]$-orbit segment of the point $\mathbf{x}_j$.
 \end{enumerate}
 We say that the $n$-tuple $$\mathcal{S}= (\pi_{[k_1,l_1]}(\mathbf{x}_1),\pi_{[k_2,l_2]}(\mathbf{x}_2), \ldots ,\pi_{[k_n,l_n]}(\mathbf{x}_n))$$ is an $n$\emph{-CR specification} or a \emph{CR-specification} in $X_F^+$.
 \end{definition}

 \begin{definition}
 Let $(X,F)$ be a CR-dynamical system, let $N$ be a positive integer, and let $\mathcal{S}= (\pi_{[k_1,l_1]}(\mathbf{x}_1),\pi_{[k_2,l_2]}(\mathbf{x}_2), \ldots ,\pi_{[k_n,l_n]}(\mathbf{x}_n))$ be a CR-specification in $X_F^+$. We say that $\mathcal{S}$ is an $N$\emph{-spaced specification} if, for each $j \in \{1,2,\ldots,n-1\}$, $$k_{j+1} - l_j \ge N.$$

 \end{definition}
 
 \begin{definition} 
 Let $(X,F)$ be a CR-dynamical system,  let $d$ be the metric on $X$, let $N$ be a positive integer, let $\epsilon >0$, let $\mathbf{y} \in X_F^+$, and let $$\mathcal{S}= (\pi_{[k_1,l_1]}(\mathbf{x}_1),\pi_{[k_2,l_2]}(\mathbf{x}_2), \ldots ,\pi_{[k_n,l_n]}(\mathbf{x}_n))$$ be an $N$\emph{-spaced specification} in $X_F^+$. We say that $\mathcal{S}$ is $\epsilon$\emph{-traced by} $\mathbf{y}$ if for each $i\in \{1,2,\ldots\,n \}$, and for each $j \in \{k_i,k_{i+1}, \ldots,l_i \}$, $$d(\mathbf{x}_i(j), \mathbf{y}(j)) \le \epsilon.$$

 \end{definition}
 
 \begin{definition} Let $F \subseteq X \times X$ be a closed relation on $X$. We say that $X_F^+$ has the \emph{CR-specification property} if for each $\epsilon >0$, there is a positive integer $N$ such that for any $N$-spaced specification  $\mathcal{S}$ in $X_F^+$, there is $\mathbf{y} \in X_F^+$ such that $\mathcal{S}$ is $\epsilon$-traced in $X_F^+$ by $\mathbf{y}$.
 	
 \end{definition}
 
 \begin{theorem} \label{CR-spec} Let $F \subseteq X \times X$ be a closed relation on $X$. The shift map $\sigma^+_F:X_F^+ \to X_F^+$ has the specification property if and only if $X_F^+$ has the CR-specification property.  	
 \end{theorem}
 
 \begin{proof} Assume $F \subseteq X \times X$ is a closed relation on $X$ and $X_F^+$ has the CR-specification property. Let $\epsilon >0$. Then there is an $M \in \mathbb{N}$ such that every $M$-spaced CR-specification in $X_F^+$ is $\epsilon$-traced by some $\mathbf{y} \in X_F^+$. Choose $N \in \mathbb{N}$ such that $\frac{1}{2^N} < \epsilon$. 
 
 \vphantom{} 
 
 Let 
 $\mathcal{S}= (\sigma_F ^{+[k_1,l_1]}(\mathbf{x}_1),\sigma_F^{+[k_2,l_2]}(\mathbf{x}_2), \ldots ,\sigma_F^{+[k_n,l_n]}(\mathbf{x}_n))$ be an $(M+N)$-spaced specification in the dynamical system $(X_F^+,\sigma_F^+)$. Let 
 
  $$\mathcal{S}^*= (\pi_{[k_1,l_1+N]}(\mathbf{x}_1),\pi_{[k_2,l_2+N]}(\mathbf{x}_2), \ldots ,\pi_{[k_n,l_n+N]}(\mathbf{x}_n)).$$ Then $\mathcal{S}^*$ is an $M$-spaced CR-specification in $X_F^+$. So there exists $\mathbf{y} \in X_F^+$ such that $\mathcal{S}^*$ is $\epsilon$-traced by $\mathbf{y}$. So, for each $i \in \{1,2, \ldots,n \}$ and each $j \in \{k_i, k_{i+1}, \ldots, l_i \}$ and for each $m \in \{0,1,2, \ldots, N \}$, we have  $d(\mathbf{x}_i(j+m), \mathbf{y}(j+m)) < \epsilon$.
  
  	Therefore, for each $i\in \{ 1,2,\ldots,n\}$, 
  \begin{align*}
  	D(\sigma_F^{+ j} (\xx_i), \sigma_F^{+ j} (\yy)) 
  	& = \sup\left\{\frac{d(\xx_i (j+m),\yy(j+m) )}{2^m} \colon m\geq 1\right\}\\ 
  	& \leq \max \left\{ \sup\left\{\frac{d(\xx_i (j+m),\yy(j+m) )}{2^m} \colon 1 \leq m< N\right\} ,\e \right\}\\
  	&\leq \max \left\{ \sup\left\{\frac{\e}{2^m} \colon 1 \leq m< N\right\} ,\e \right\} = \e.
  \end{align*}
  Therefore, $\mathcal{S}$ is $\e$-traced by $\yy \in X_F^+.$
  
  Now assume $\sigma_F^+ : X_F^+ \to X_F^+$ has the specification property. Let $\e >0.$  
  Let $N \in \N$ such that every $N$-spaced specification in  $(X_F^+,\sigma_F^+)$ can be $\e/2$-traced by some $\yy \in X_F^+$.
  Let $$\mathcal{S}=(\pi_{[k_1,\ell_1]}(\xx_1),\pi_{[k_2,\ell_2]}(\xx_2),\ldots,\pi_{[k_n,\ell_n]}(\xx_n))$$ be an $N$-spaced CR-specification in $X_F^+$. Then 
  $$\mathcal{S^*}=(\sigma_F^{+ {[k_1-1,\ell_1-1]}}(\xx_1),\sigma_F^{+ {[k_2-1,\ell_2-1]}}(\xx_2),\ldots,\sigma_F^{+ {[k_n-1,\ell_n-1]}}(\xx_n))$$
  is an $N$-spaced specification in $(X_F^+,\sigma_F^+).$ So there is a $\yy \in X_F^+$ which $\e/2$-traces $\mathcal{S^*}$. But then for each   $i\in \{ 1,2,\ldots,n\}$ and each $j \in \{ k_i,k_{i+1},\ldots,l_i\},$
  \begin{align*}
  	d(\xx_i(j),\yy(j))) 
  	&\leq  2\sup\left\{\frac{d(\xx_i (j-1+m),\yy(j-1+m) )}{2^m} \colon m\geq 1\right\} \\ 
  	& = 2 D(\sigma_F^{+ (j-1)} (\xx_i), \sigma_F^{+ (j-1)} (\yy))<\e.
  \end{align*}
  So, $X_F^+$ has the CR-specification property.

 \end{proof}
 
 \begin{prop} \label{Proposition: Mixing implies growing images}
 	Let $(X,F)$ be a CR-dynamical system, and let $(X_F,\sigma_F)$ and $(X_F^+,\sigma^+_F)$ be the corresponding Mahavier systems. If $(X_F,\sigma_F)$ (or $(X_F^+,\sigma_F^+)$) is topologically mixing, then for any closed subset $A \subseteq X$ with non-empty interior, 
 $$\underset{n \to \infty}\lim d_H(F^n(A),X)=0,$$ where $d_H$ represents the Hausdorff metric on the hyperspace $2^X$.
 	
 \end{prop}

\begin{proof}
	Suppose that $(X_F,\sigma_F)$ is topologically mixing. Let $A\se X$ be a closed subset of $X$ with non-empty interior, and let $\e>0$. Since $X$ is compact, we may choose $\mathcal{U}=\{U_1,\ldots,U_k\}$ to be a finite open cover of $X$ such that for each $j\in\{1,\ldots,k\}$, $\diam(U_j)<\e$.
	
	Since $A$ has non-empty interior, there is a non-empty open set $V\se A$. Define the following open sets:
	
	\begin{align*}
		\widetilde{V}&=\pi_0^{-1}(V)=\left\{\xx\in I_F\colon x_0\in V\right\},\text{ and}\\
		\widetilde{U}_j&=\pi_0^{-1}\left(U_j\right)=\left\{\xx\in I_F\colon x_0\in U_j\right\}\text{ for all }j\in\{1,\ldots,k\}.
	\end{align*}
	By assumption $(X_F,\sigma_F)$ is topologically mixing, so for all $j\in\{1,\ldots,k\}$ there exists a positive integer $N_j$ such that for all $n\geq N_j$, $\sigma_F^n(\widetilde{V})\cap\widetilde{U}_j\neq\emptyset$. Choose $N=\max\{N_j\colon 1\leq j\leq k\}$. Then for all $n\geq N$, $\sigma_F^n(\widetilde{V})\cap \widetilde{U}_j\neq\emptyset$ for all $j\in\{1,\ldots,k\}$.
	
	Given $n\geq N$ and $j\in\{1,\ldots,k\}$ there exists $\xx(n,j)\in \widetilde{V}$ such that $\sigma_F^n(\xx(n,j))\in\widetilde{U}_j$. From the definitions of $\widetilde{V}$ and $\widetilde{U}_j$ and the definition of the shift map $\sigma_F$, it follows that $x_0(n,j)\in V\se A$ and $x_n(n,j)\in U_j$. Observe that since $\xx(n,j)\in I_F$, we must have that $x_n(n,j)\in F^n(x_0(n,j))\se F^n(A)$. 
	
	Now we show this implies that for $n\geq N$, $d_H(F^n(A),X)<\e$. Fix $n\geq N$, and let $y\in X$. Then $y\in U_j$ for some $j\in\{1,\ldots,k\}$. As we showed, there exists a point $x_n(n,j)\in F^n(A)\cap U_j$. Since $\diam(U_j)<\e$, we have that $d(y,x_n(n,j))<\e$. We have thus shown that every element of $X$ is within $\e$ of some element of $F^n(A)$. Since $F^n(A)\se X$, this suffices to show that
	\[
	d_H\left(F^n(A),X\right)<\e,
	\]
	and therefore
	\[
	\lim_{n\to\infty}d_H(F^n(A),X)=0.
	\]
	Using the similar arguments obtain the result that $(X_F^+,\sigma_F^+)$ is topologically mixing.
\end{proof}

\section{Lelek fan via Mahavier products}

 The Lelek fan (constructed in \cite{lelek}) is the unique smooth fan with a dense set of endpoints (see \cite{Char}, \cite{BO}). It has an astonishing way of appearing in different contexts: in continuum theory, it is a canonical example of a highly structured and yet intricate dendroid. Surprisingly, its endpoints form a one-dimensional set. It is not homogeneous, but it exhibits strong approximate homogeneity as its homeomorphism group acts densely on the endpoint set (see \cite{BK2}). It can be realized as the Fraiss\'e limit of finite fans (see \cite{BK}). Consequently, the Lelek fan serves as a fundamental test object linking continuum theory, descriptive topology, and topological dynamics. 
 
 We note that, using completely different techniques P. Oprocha \cite{Oprocha} has shown that the Lelek fan admits a completely scrambled, weakly mixing homeomorphism, and P. Oprocha and V. Nall have shown that it admits a mixing homeomorphism with zero entropy. 
 
 Here we show it admits still more dynamical properties. 

\begin{definition} \cite{Char}, \cite{BO}
	A Lelek fan is the unique smooth fan with a dense set of endpoints. 
\end{definition}

For our work here, we need the Lelek fan constructed in a particular way (for more details see \cite{banic1}).

\begin{definition} For each $(r,\rho) \in (0,\infty)$, we define the sets $L_r$, $L_{\rho}$, and $L_{r,\rho}$ as follows: $L_r= \{(x,y) \in I \times I:y=rx \}$, $L_{\rho}= \{(x,y) \in I \times I:y=\rho x \}$, and $L_{r,\rho}= L_r \cup L_{\rho}$.
	
\end{definition}

\begin{definition} Let $(r,\rho) \in (0,\infty) \times (0,\infty)$. We say that $r$ and $\rho$ \emph{never connect} or $(r,\rho) \in \mathcal{NC}$ if 
\begin{enumerate}
\item $0<r<1$, $\rho >1	$, and
\item for all integers $k$ and $l$, $$r^k=\rho^l \text{ if and only if } k=l=0.$$
\end{enumerate}
	
\end{definition}

In \cite{banic1}, the following theorem is the main result:

\begin{theorem} \cite{banic1} Let $(r,\rho) \in \mathcal{NC}$. Then $X^+_{L_{r,\rho}}$ is a Lelek fan with top $(0,0,0,\ldots)$.	
\end{theorem}

\begin{definition} Let $(r,\rho) \in (0,\infty) \times (0,\infty)$. We use $F_{r,\rho}$ to denote the following closed relation on $X$. 
$$F_{r,\rho}=L_{r,\rho} \cup \{(t,t):0\le t \le 1 \}.$$
	
\end{definition}

The Mahavier products generated by $F_{r,\rho}$ are also Lelek fans - this was proven in \cite{BE2}.

\begin{theorem} \cite{BE2} Let $(r,\rho) \in \mathcal{NC}$. Then $X_{F_{r,\rho}}$ and $X_{F_{r,\rho}}^+$ are both Lelek fans.
	
\end{theorem}

We next extend these results to even more straight line segments. In order to do this we need to generalize and uniformize our notation. We do this and then introduce a family of closed relations and show that all of them have Lelek fans as their Mahavier products. 

\begin{definition}\label{Definition: LF-inducing}
	If $M$ is an integer with $M \ge 2$, and $\Omega= \{\omega_1, \omega_2, \ldots, \omega_M \}$ is a set of positive numbers, then we say that $\Omega$ is \emph{LF-inducing} if 
	\begin{enumerate}
		\item $\omega_i > 0$ for all $i \in \{1,2, \ldots, M\}$, and 
		\item $\omega_1<1<\omega_2$, and the pair $(\omega_1,\omega_2) \in \mathcal{NC}$.	
	\end{enumerate}
	
	Given an LF-inducing set $\Omega= \{\omega_1, \omega_2, \ldots, \omega_M \}$ we define the relation $F_{\Omega} \subseteq I^2$ by
	$$F_{\Omega}= \bigcup^M_{j=1}\{(x,y)\in I^2:y=\omega_jx \}.$$
	Figures 3, 4, and 5 depict the relations $F_{\{1/2,3,1/3,2,1\}}$, $F_{\{1/2,3,1/3,2\}}$, and $F_{\{1/2,3\}}$.
	
\end{definition}

\begin{center}
	\begin{figure}
		\begin{minipage}{.49\textwidth}
			\centering
			\begin{tikzpicture}[scale=5]
				\draw[thick,dotted] (0,0)node[left]{\footnotesize0} -- (0,1)node[left]{\footnotesize1} -- (1,1) -- (1,0)node[below]{\footnotesize1} -- 	(0,0)node[below]{\footnotesize0};
				\draw[thick, dotted] (1/3,1) -- (1/3,0) node[below]{\footnotesize$\frac{1}{3}$}
				(1/2,1) -- (1/2,0) node[below]{\footnotesize$\frac{1}{2}$}
				(1,1/2) -- (0,1/2) node[left]{\footnotesize$\frac{1}{2}$}
				(1,1/3) -- (0,1/3) node[left]{\footnotesize$\frac{1}{3}$}
				;
				\draw[very thick] 
				(0,0) -- (1,1)
				(0,0) -- (1/3,1)
				(0,0) -- (1/2,1)
				(0,0) -- (1,1/2)
				(0,0) -- (1,1/3)
				;
			\end{tikzpicture}
			\caption{$F_{\{1/2,3,1,1/3,2\}}$}\label{Figure: 5 Lines}
		\end{minipage}		
		\begin{minipage}{.49\textwidth}
			\centering
			\begin{tikzpicture}[scale=5]
				\draw[thick,dotted] (0,0)node[left]{\footnotesize0} -- (0,1)node[left]{\footnotesize1} -- (1,1) -- (1,0)node[below]{\footnotesize1} -- 	(0,0)node[below]{\footnotesize0};
				\draw[thick, dotted] (1/3,1) -- (1/3,0) node[below]{\footnotesize$\frac{1}{3}$}
				(1/2,1) -- (1/2,0) node[below]{\footnotesize$\frac{1}{2}$}
				(1,1/2) -- (0,1/2) node[left]{\footnotesize$\frac{1}{2}$}
				(1,1/3) -- (0,1/3) node[left]{\footnotesize$\frac{1}{3}$}
				;
				\draw[very thick] 
				(0,0) -- (1/3,1)
				(0,0) -- (1/2,1)
				(0,0) -- (1,1/2)
				(0,0) -- (1,1/3)
				;
			\end{tikzpicture}
			\caption{$F_{\{1/2,3,1/3,2\}}$}\label{Figure: 4 Lines}
		\end{minipage}
		%
		\centering
		\begin{tikzpicture}[scale=5]
			\draw[thick,dotted] (0,0)node[left]{\footnotesize0} -- (0,1)node[left]{\footnotesize1} -- (1,1) -- (1,0)node[below]{\footnotesize1} -- (0,0)node[below]{\footnotesize0};
			\draw[thick, dotted] (1/3,1) -- (1/3,0) node[below]{\footnotesize$\frac{1}{3}$}
			(1,1/2) -- (0,1/2) node[left]{\footnotesize$\frac{1}{2}$}
			;
			\draw[very thick] 
			(0,0) -- (1/3,1)
			(0,0) -- (1,1/2)
			;
		\end{tikzpicture}
		\caption{$F_{\{1/2,3\}}$}\label{Figure: 2 Lines}
	\end{figure}
\end{center}

\begin{lemma}\label{Lemma: G_alpha yields a Lelek fan}
	Let $(\omega_1,\omega_2)\in\mathcal{NC}$, and let $\alpha\in(0,1]$. If $G_\alpha=F_{\{\omega_1,\omega_2\}}\cap\left([0,\alpha]\times[0,\alpha]\right)$, then $[0,\alpha]^+_{G_\alpha}$ and $[0,\alpha]^+_{G_\alpha^{-1}}$ are Lelek fans.
\end{lemma}

\begin{proof}
	From \cite{banic1} we have that both $I_{F_{\{\omega_1,\omega_2\}}}^+$ and $I_{F^{-1}_{\{\omega_1,\omega_2\}}}^+$ are Lelek fans. Let $\alpha\in(0,1]$. Define $\varphi_\alpha\colon[0,1]\to[0,\alpha]$ by $\varphi_\alpha(x)=\alpha x$. This is clearly a homeomorphism. Moreover, given $(x,y)\in I\times I$, we have $y=\omega_1x$ if and only if $\alpha y=\omega_1(\alpha x)$, and likewise $y=\omega_2x$ if and only if $y=\omega_2(\alpha x)$. Thus $(x,y)\in F_{\{\omega_1,\omega_2\}}$ if and only if $(\varphi_\alpha(x),\varphi_\alpha(y))\in G_\alpha$, so $\varphi_\alpha$ is a topological conjugacy from $([0,1],F_{\{\omega_1,\omega_2\}})$ to $([0,\alpha],G_\alpha)$.  Therefore $[0,\alpha]_{G_\alpha}^+$ is a Lelek fan.
	
	Similarly, $([0,\alpha],G_\alpha^{-1})$ is topologically conjugate to $([0,\alpha],F_{\{\omega_1,\omega_2\}}^{-1})$, so $[0,\alpha]_{G_\alpha^{-1}}^+$ is a Lelek fan.
\end{proof}

\begin{lemma}\label{Lemma: I_F^+ is a smooth fan}
	If $\Omega=\left\{\omega_1,\ldots,\omega_M\right\}$ is an LF-inducing set, then $I_{F_\Omega}^+$ is a smooth fan.
\end{lemma}
\begin{proof}
	For each sequence of indices, $\mathbf{a}=(a(1),a(2),\ldots)\in\{1,\ldots,M\}^\N$, define the sets
	\begin{align*}
		B_{\mathbf{a}}&=\left\{\xx\in I^\N\colon x_{j+1}=\omega_{a(j)}x_j\text{ for all }j\in\N\right\},\\
		C_{\mathbf{a}}&=B_{\mathbf{a}}\cap I_{F_\Omega}^+.
	\end{align*}
	The set $B_{\mathbf{a}}$ is an arc, and we have that $C_{\mathbf{a}}$ is either an arc or the singleton $\{(0,0,\ldots)\}$. Additionally, given any two sequences $\mathbf{a},\mathbf{b}\in\{1,\ldots,M\}^\N$ with $\mathbf{a}\neq\mathbf{b}$, we have that $C_{\mathbf{a}}\cap C_{\mathbf{b}}=\{(0,0,\ldots)\}$.
	
	Observe that
	\[
	I_{F_\Omega}^+=\bigcup_{\mathbf{a}\in \{1,\ldots,M\}^\N} C_{\mathbf{a}},
	\]
	so $I_{F_\Omega}^+$ is a fan with top $\mathbf{v}=(0,0,0,\ldots)$.
	
	To see that $I_{F_\Omega}^+$ is smooth, let $\mathbf{y}\in I_{F_\Omega}^+\setminus\{\mathbf{v}\}$ and suppose $(\mathbf{z}(n))_{n=1}^\infty$ is a sequence in $I_{F_\Omega}^+$ converging to $\mathbf{y}$. Without loss of generality, we suppose that for all $n\in\N$, $\mathbf{z}(n)\neq\mathbf{v}$. Then for each $n\in\N$, there is a unique sequence $\mathbf{b}(n)=(b(n,1),b(n,2),\ldots)\in\{1,\ldots,M\}^\N$ such that $\mathbf{z}(n)\in C_{\mathbf{b}(n)}$, and there exists a unique $\mathbf{a}\in\{1,\ldots,M\}^\N$ such that $\mathbf{y}\in C_{\mathbf{a}}$.
	
	Define the functions $\psi\colon[0,y_1]\to C_{\mathbf{b}}$ and $\psi_n\colon[0,z(n,1)]\to C_{\mathbf{b}(n)}$ for each $n\in\N$ by
	\begin{align*}
		\psi(t)&=\left(t,\omega_{a(1)}t,\omega_{a(2)}\omega_{a(1)}t,\omega_{a(3)}\omega_{a(2)}\omega_{a(1)}t,\ldots\right)\\
		\psi_n(t)&=\left(t,\omega_{b(n,1)}t,\omega_{b(n,2)}\omega_{b(n,1)}t,\omega_{b(n,3)}\omega_{b(n,2)}\omega_{b(n,1)}t,\ldots\right).
	\end{align*}
	Then the image $\psi([0,y_1])$ is the unique arc in $I_{F_\Omega}^+$ from $\mathbf{v}$ to $\mathbf{y}$, and the image $\psi_n([0,z(n,1)])$ is the unique arc in $I_{F_\Omega}^+$ from $\mathbf{v}$ to $\mathbf{z}(n)$. By assumption $(\mathbf{z}(n))_{n=1}^\infty$ converges to $\mathbf{y}$, so we may look inductively at each coordinate to conclude that for each $i\in\N$, $(b(n,i))_{n=1}^\infty$ converges to $a(i)$. We thus conclude that the arcs $\psi_n([0,z(n,1)])$ converge to the arc $\psi([0,y_1])$. Therefore $I_{F_\Omega}^+$ is a smooth fan.
\end{proof}

We establish some notation that will be used in Lemma~\ref{Lemma: L_gamma^+ is a Lelek fan} and Theorem~\ref{Proposition: LF-inducing yields Lelek fan (one-sided)}. Given a finite sequence of indices, $\gamma=(a(1),\ldots,a(N-1))\in \{1,\ldots,M\}^{N-1}$, we define
\[
A_\gamma=\left\{\left(x_1,\ldots, x_{N}\right)\in I^{N}\colon x_{j+1}=\omega_{a(j)}x_j\text{ for all }j\in\{1,\ldots,N-1\}\right\}.
\]
Observe that $A_{\gamma}$ is an arc. Let $\alpha(\gamma)$ be the real number such that $\pi_{n}(A_\gamma)=[0,\alpha(\gamma)]$. Observe that $\alpha(\gamma)\in (0,1]$. As in Lemma~\ref{Lemma: G_alpha yields a Lelek fan}, let
\[
G_{\alpha(\gamma)}=F_{\{\omega_1,\omega_2\}}\cap\left([0,\alpha(\gamma)]\times[0,\alpha(\gamma)]\right),
\]
and define
\[
L_\gamma^+=A_\gamma\star [0,\alpha(\gamma)]_{G_{\alpha(\gamma)}}^+.
\]

\begin{lemma}\label{Lemma: L_gamma^+ is a Lelek fan}
	Let $\Omega=\left\{\omega_1,\ldots,\omega_M\right\}$ be an LF-inducing set. For every $N\geq 2$ and every $\gamma=(a(1),a(2),\ldots,a(N-1))\in\{1,\ldots,M\}^{N-1}$, the set $L_\gamma^+$ is a Lelek fan with top $\mathbf{v}=(0,0,\ldots)$, and every endpoint of $L_\gamma^+$ is an endpoint of $I_{F_\Omega}^+$.
\end{lemma}

\begin{proof}
	We have from Lemma~\ref{Lemma: G_alpha yields a Lelek fan} that $[0,\alpha(\gamma)]_{G_{\alpha(\gamma)}}^+$ is a Lelek fan. We can define a function $h\colon [0,\alpha(\gamma)]_{G_{\alpha(\gamma)}}^+\to L_\gamma^+$ as follows: given $\xx=(x_1,x_2,\ldots)\in [0,\alpha(\gamma)]_{G_{\alpha(\gamma)}}^+$, we define
	\[
	h(\xx)=\left(\frac{x_1}{\omega_{a(1)}\cdots\omega_{a(n-2)}\omega_{a(n-1)}},\ldots,\frac{x_1}{\omega_{a(n-2)}\omega_{a(n-1)}},\frac{x_1}{\omega_{a(n-1)}},x_1,x_2,x_3\ldots\right).
	\]
	Since $x_1\in[0,\alpha(\gamma)]$ and $\alpha(\gamma)=\pi_n(A_\gamma)$, we have that $h(\xx)\in I^\N$, so $h(\xx)\in L_\gamma$.
	
	The function $h$ is invertible with $h^{-1}(x_1,x_2,\ldots)=(x_{n},x_{n+1},\ldots)$, and clearly both $h$ and $h^{-1}$ are continuous. Thus $h$ is a homeomorphism, so $L_\gamma^+$ is a Lelek fan. 
\end{proof}

\begin{theorem}\label{Proposition: LF-inducing yields Lelek fan (one-sided)}
	If $\Omega=\left\{\omega_1,\ldots,\omega_M\right\}$ is an LF-inducing set, then $I_{F_\Omega}^+$ is a Lelek fan.
\end{theorem}
\begin{proof}
	By Lemma~\ref{Lemma: I_F^+ is a smooth fan}, $I_{F_\Omega}^+$ is a smooth fan, so we only need to show that the set of endpoints of $I_{F_\Omega}^+$ is dense in $I_{F_\Omega}^+$. Let $\mathbf{x}\in I_{F_\Omega}^+$, and let $\e>0$. Fix $N\in\N$ such that $2^N<\e$. Since $\mathbf{x}\in I_{F_\Omega}^+$, there exists a sequence $\mathbf{a}\in\{1,\ldots,M\}^\N$ such that $x_{j+1}=\omega_{a(j)}x_{j}$ for every $j\in\N$. Let $\gamma=(a(1),\ldots,a(N-1))$. Then $L_\gamma^+$ is a Lelek fan by Lemma~\ref{Lemma: L_gamma^+ is a Lelek fan}. Fix a point $\mathbf{y}\in L_\gamma^+$ such that $y_j=x_j$ for all $j\in\{1,\ldots,N\}$. Then $D(\mathbf{x},\mathbf{y})\leq 1/2^{N+1}<\e/2$.
	
	Since $\mathbf{y}\in L_\gamma^+$, there exists an endpoint $\mathbf{e}\in L_\gamma$ such that $D(\mathbf{y},\mathbf{e})<\e/2$. Note that $\mathbf{e}$ is also an endpoint of $I_{F_\Omega}^+$, and we have
	\[
	D(\mathbf{x},\mathbf{e})<D(\mathbf{x},\mathbf{y})+D(\mathbf{y},\mathbf{e})<\frac{\e}{2}+\frac{\e}{2}=\e.
	\]
	Therefore, the set of endpoints of $I_{F_\Omega}^+$ is dense in $I_{F_\Omega}^+$, so $I_{F_\Omega}^+$ is a Lelek fan.
\end{proof}

We now turn our attention to $I_{F_\Omega}$. The proofs for $I_{F_\Omega}$ are nearly identical to the respective proofs for $I_{F_\Omega}^+$.

\begin{lemma}\label{Lemma: I_F is a smooth fan}
	If $\Omega=\left\{\omega_1,\ldots,\omega_M\right\}$ is an LF-inducing set, then $I_{F_\Omega}$ is a smooth fan.
\end{lemma}
\begin{proof}
	For each $\mathbf{a}=(\ldots,a(-2),a(-1),a(0),a(1),a(2),\ldots)\in\{1,\ldots,M\}^\Z$, define the sets
	\begin{align*}
		B_{\mathbf{a}}&=\left\{\xx\in I^\Z\colon x_{j+1}=\omega_{a(j)}x_j\text{ for all }j\in\Z\right\},\\
		C_{\mathbf{a}}&=B_{\mathbf{a}}\cap I_{F_\Omega}.
	\end{align*}
	The set $B_{\mathbf{a}}$ is an arc, and we have that $C_{\mathbf{a}}$ is either an arc or the singleton $\{(0,0,\ldots)\}$. Additionally, given any two sequences $\mathbf{a},\mathbf{b}\in\{1,\ldots,M\}^\Z$ with $\mathbf{a}\neq\mathbf{b}$, we have that $C_{\mathbf{a}}\cap C_{\mathbf{b}}=\{(0,0,\ldots)\}$.
	
	Observe that
	\[
	I_{F_\Omega}=\bigcup_{\mathbf{a}\in \{1,\ldots,M\}^\Z} C_{\mathbf{a}},
	\]
	so $I_{F_\Omega}$ is a fan with top $\mathbf{v}=(0,0,0,\ldots)$.
	
	To see that $I_{F_\Omega}$ is smooth, let $\mathbf{y}\in I_{F_\Omega}\setminus\{\mathbf{v}\}$ and suppose $(\mathbf{z}(n))_{n=1}^\infty$ is a sequence in $I_{F_\Omega}$ converging to $\mathbf{y}$. Without loss of generality, we suppose that for all $n\in\Z$, $\mathbf{z}(n)\neq\mathbf{v}$. Then for each $n\in\Z$, there is a unique sequence $\mathbf{b}(n)=(b(n,1),b(n,2),\ldots)\in\{1,\ldots,M\}^\Z$ such that $\mathbf{z}(n)\in C_{\mathbf{b}(n)}$, and there exists a unique $\mathbf{a}\in\{1,\ldots,M\}^\Z$ such that $\mathbf{y}\in C_{\mathbf{a}}$.
	
	Define the functions $\psi\colon[0,y_1]\to C_{\mathbf{b}}$ and $\psi_n\colon[0,z(n,1)]\to C_{\mathbf{b}(n)}$ for each $n\in\Z$ as follows:
	\begin{itemize}
		\item Given $t\in[0,y_0]$, we set $\psi(t)=\mathbf{u}$ where 
		\begin{itemize}
			\item $u_0=t$; 
			\item For all $j\geq 0$, $u_{j+1}=\omega_{a(j)}u_j$;
			\item For all $j\leq 0$, $u_{j-1}=u_j/\omega_{a(j-1)}$.
		\end{itemize}
		\item Given $t\in[0,z(n,0)]$, we set $\psi_n(t)=\mathbf{w}$ where
		\begin{itemize}
			\item $u_0=t$; 
			\item For all $j\geq 0$, $u_{j+1}=\omega_{b(n,j)}w_j$;
			\item For all $j\leq 0$, $w_{j-1}=w_j/\omega_{b(n,j-1)}$.
		\end{itemize}
	\end{itemize}
	
	Then the image $\psi([0,y_0])$ is the unique arc in $I_{F_\Omega}$ from $\mathbf{v}$ to $\mathbf{y}$, and the image $\psi_n([0,z(n,0)])$ is the unique arc in $I_{F_\Omega}$ from $\mathbf{v}$ to $\mathbf{z}(n)$. By assumption $(\mathbf{z}(n))_{n=1}^\infty$ converges to $\mathbf{y}$, so we may look inductively at each coordinate to conclude that for each $i\in\Z$, $(b(n,i))_{n=1}^\infty$ converges to $a(i)$. We thus conclude that the arcs $\psi_n([0,z(n,0)])$ converge to the arc $\psi([0,y_0])$. Therefore $I_{F_\Omega}$ is a smooth fan.
\end{proof}

We establish additional notation that will be used in Lemma~\ref{Lemma: L_gamma is a Lelek fan} and Theorem~\ref{Proposition: LF-inducing yields Lelek fan (two-sided)}. Given a finite sequence of indices $\gamma=(a(-N),\ldots,a(0),\ldots,a(N-1))\in\{1,\ldots,M\}^{2N}$, we define
\[
A_\gamma=\left\{\left(x_{-N},\ldots,x_0,\ldots, x_{N}\right)\in I^{2N+1}\colon x_{j+1}=\omega_{a(j)}x_j\text{ for all }j\in\{-N,\ldots,N-1\}\right\}.
\]
Observe that $A_{\gamma}$ is an arc.  Let $\alpha(\gamma)$ and $\beta(\gamma)$ be the real numbers such that $\pi_{-N}(A_\gamma)=[0,\alpha(\gamma)]$ and $\pi_N(A_\gamma)=[0,\beta(\gamma)]$. Observe that $\alpha(\gamma),\beta(\gamma)\in (0,1]$. 

As before, let
\begin{align*}
	G_{\alpha(\gamma)}&=F_{\{\omega_1,\omega_2\}}\cap\left([0,\alpha(\gamma)]\times[0,\alpha(\gamma)]\right),\text{ and }\\
	G_{\beta(\gamma)}&=F_{\{\omega_1,\omega_2\}}\cap\left([0,\beta(\gamma)]\times[0,\beta(\gamma)]\right).
\end{align*}
and define
\[
L_\gamma=[0,\alpha(\gamma)]_{G_{\alpha(\gamma)}^{-1}}^+\star A_\gamma\star[0,\beta(\gamma)]_{G_{\beta(\gamma)}}^+.
\]

\begin{lemma}\label{Lemma: L_gamma is a Lelek fan}
	Let $\Omega=\left\{\omega_1,\ldots,\omega_M\right\}$ be an LF-inducing set. For every $N\geq 1$ and every $\gamma=(a(-N),\ldots,a(N-1))\in\{1,\ldots,M\}^{2N}$, the set $L_\gamma$ is a Lelek fan with top $\mathbf{v}=(0,0,\ldots)$, and every endpoint of $L_\gamma$ is an endpoint of $I_{F_\Omega}$.
\end{lemma}

\begin{proof}
	We show that $L_\gamma$ is topologically conjugate to $I_{F_{\{\omega_1,\omega_2\}}}$ which is known to be a Lelek fan. We define $\Phi\colon I_{F_{\{\omega_1,\omega_2\}}}\to L_\gamma$  as follows: given $\xx=(x_j)_{j\in\Z}\in I_{F_{\{\omega_1,\omega_2\}}}$, we set $\Phi(\xx)=\yy$ where
	\begin{itemize}
		\item for $j\geq N$, $y_j=\beta(\gamma)x_j$,
		\item for $j\in\{-(N-1),-(N-2),\ldots,N-2,N-1\}$, $y_j=y_{j+1}/\omega_{a(j)}$, and
		\item for $j\leq -N$, $y_j=\alpha(\gamma)x_j$.
	\end{itemize}
	Then $\Phi$ is a topological conjugacy between $I_{F_{\{\omega_1,\omega_2\}}}$ and $L_\gamma$, so $L_\gamma$ is a Lelek fan.
\end{proof}

\begin{theorem}\label{Proposition: LF-inducing yields Lelek fan (two-sided)}
	If $\Omega=\left\{\omega_1,\ldots,\omega_M\right\}$ is an LF-inducing set, then $I_{F_\Omega}$ is a Lelek fan.
\end{theorem}
\begin{proof}
	By Lemma~\ref{Lemma: I_F is a smooth fan}, $I_{F_\Omega}$ is a smooth fan, so we only need to show that the set of endpoints of $I_{F_\Omega}$ is dense in $I_{F_\Omega}$. Let $\mathbf{x}\in I_{F_\Omega}$, and let $\e>0$. Fix $N\in\N$ such that $2^N<\e$. Since $\mathbf{x}\in I_{F_\Omega}$, there exists a sequence $\mathbf{a}\in\{1,\ldots,M\}^\Z$ such that $x_{j+1}=\omega_{a(j)}x_{j}$ for every $j\in\Z$. Let $\gamma=(a(-N),\ldots,a(N-1))$. Then $L_\gamma$ is a Lelek fan by Lemma~\ref{Lemma: L_gamma is a Lelek fan}. Fix a point $\mathbf{y}\in L_\gamma$ such that $y_j=x_j$ for all $j\in\{-N,\ldots,N\}$. Then $D(\mathbf{x},\mathbf{y})\leq 1/2^{N+1}<\e/2$.
	
	Since $\mathbf{y}\in L_\gamma$, there exists an endpoint $\mathbf{e}\in L_\gamma$ such that $D(\mathbf{y},\mathbf{e})<\e/2$. Note that $\mathbf{e}$ is also an endpoint of $I_{F_\Omega}$, and we have
	\[
	D(\mathbf{x},\mathbf{e})<D(\mathbf{x},\mathbf{y})+D(\mathbf{y},\mathbf{e})<\frac{\e}{2}+\frac{\e}{2}=\e.
	\]
	Therefore, the set of endpoints of $I_{F_\Omega}$ is dense in $I_{F_\Omega}$, so $I_{F_\Omega}$ is a Lelek fan.
\end{proof}

\section{Examples}

 Finally, we move to the main goal of our paper: the investigation of the dynamical properties admitted by our LF-inducing systems. 
 
\subsection{Two line example}

In this subsection, we consider the relation $F_{\{1/2,3\}} \subseteq [0,1]^2$, which is defined by 

$$F_{\{1/2,3\}}= \left\{ (x,3x):x \in \left[0,\frac{1}{3}\right] \bigg\} \cup \bigg\{\left(x,\frac{x}{2}\right): x \in [0,1] \right\}. $$

\begin{remark} The choice of $\frac{1}{2}$ and $3$ is not important - any never connecting pair $(r,\rho)$ yields the same result, i.e., that $X_{F_{\{r,\rho\}}}^+$ and $X_{F_{\{r,\rho\}}}$ are Lelek fans.
\end{remark}

It was shown in \cite{banic2} that $(I_{F_{\{1/2,3\}}}^+,\sigma_{F_{\{1/2,3\}}}^+)$ is a transitive dynamical system. However, this system does not have topological mixing, shadowing, nor the specification property:

\begin{prop} The dynamical systems $(I_{F_{\{1/2,3\}}}^+,\sigma_{F_{\{1/2,3\}}}^+)$ and $(I_{F_{\{1/2,3\}}},\sigma_{F_{\{1/2,3\}}})$ are not topologically mixing.
	
\end{prop}

\begin{proof}
	Given a positive integer $M$, we may define the set
	\[
	\mathcal{A}(M)=\left\{\frac{3^m}{2^n}\colon n,m\geq 0,~n+m=M\right\}.
	\]
	Observe that for any interval $[c,d]\se [0,1]$, we have
	\[
	F_{\{1/2,3\}}^M([a,b])\se\bigcup_{\alpha\in\mathcal{A}(M)}[\alpha a,\alpha b].
	\]
	(Here we don't have equality because $\alpha a$ and $\alpha b$ may be greater than 1 for some values of $\alpha\in\mathcal{A}(M)$ whereas we have $F_{\{1/2,3\}}^M([a,b])\se [0,1]$.)
	
	For each positive integer $M$, define
	\[
	\alpha(M)=\max\left\{\alpha\in\mathcal{A}(M)\colon \alpha< 1\right\}\\
	\]
	Observe that there exist non-negative integers $p$ and $q$ such that $p+q=M$ and
	\[
	\alpha(M)=\frac{3^p}{2^q}.
	\]
	We may thus write the elements of $\mathcal{A}(M)$ in decreasing order as
	\begin{align*}
		\mathcal{A}(M)&=\left\{\frac{3^p}{2^q},\frac{3^{p-1}}{2^{q+1}},\ldots,\frac{3^{0}}{2^{q+p}}\right\}\\
		&=\left\{\alpha(M),\frac{1}{6}\alpha(M),\ldots,\frac{1}{6^p}\alpha(M)\right\}.
	\end{align*}
	
	Consider the interval $[\frac{5}{6},1]\se[0,1]$. Given a positive integer $M$, we have that
	\begin{align*}
		F_{\{1/2,3\}}^M\left(\left[\frac{5}{6},1\right]\right)&=\left(\bigcup_{\alpha\in\mathcal{A}(M)}\left[\frac{5}{6}\alpha,\alpha\right]\right)\cap[0,1]\\
		&=\left[\frac{5}{6}\alpha(M),\alpha(M)\right]\cup\left[\frac{5}{6}\cdot\frac{1}{6}\alpha(M),\frac{1}{6}\alpha(M)\right]\cup\cdots\cup\left[\frac{5}{6}\cdot\frac{1}{6^p}\alpha(M),\frac{1}{6^p}\alpha(M)\right].
	\end{align*}
	In particular, we have that 
	\[
	[0,1]\setminus F_{\{1/2,3\}}^M\left(\left[\frac{5}{6},1\right]\right)\supseteq\left(\frac{1}{6}\alpha(M),\frac{5}{6}\alpha(M)\right),
	\]
	so it follows that
	\begin{align*}
		d_H\left([0,1],f^M\left[\frac{5}{6},1\right]\right)&\geq\diam\left(\frac{1}{6}\alpha(M),\frac{5}{6}\alpha(M)\right)\\
		&=\left(\frac{5}{6}-\frac{1}{6}\right)\alpha(M)\\
		&=\frac{2}{3}\alpha(M).
	\end{align*}
	
	Finally, we observe that for all $M$, $\alpha(M)\geq 1/6$. This is because otherwise, we would have $6\alpha(M)<1$ which would contradict the maximality of $\alpha(M)$. Thus we have
	\[
	d_H\left([0,1],f^M\left[\frac{5}{6},1\right]\right)\geq\frac{2}{3}\alpha(M)\geq\frac{1}{9}.
	\]
	Note that this is independent of our choice of $M$, so $d_H(F_{\{1/2,3\}}^M([5/6,1]),[0,1])$ does not converge to 0. Therefore, by Proposition~\ref{Proposition: Mixing implies growing images}, $(I_{F_{\{1/2,3\}}},\sigma_{F_{\{1/2,3\}}})$ is not topologically mixing. And now, by \cite[Theorem 3.9]{BE2} $(I_{F_{\{1/2,3\}}}^+,\sigma_{F_{\{1/2,3\}}}^+)$ is not topologically mixing.
	
\end{proof}

\begin{theorem} \label{no mixing} The dynamical system $(I_{F_{\{1/2,3\}}}^+,\sigma_{F_{\{1/2,3\}}}^+)$ does not have the shadowing property.
	
\end{theorem}

\begin{proof}
	Suppose $(I^+_{F_{\{1/2,3\}}},\sigma^+_{F_{\{1/2,3\}}})$ has the shadowing property. We show this would imply $(I^+_{F_{\{1/2,3\}}},\sigma^+_{F_{\{1/2,3\}}})$ is topologically mixing. Let $U,V\se I^+_{F_{\{1/2,3\}}}$ be non-empty open sets. Fix $\aa\in U$ and $\bb\in V$, and choose $\e>0$ so that $B(\aa,\e/2)\se U$ and $B(\bb,\e/2)\se V$. Since $(I^+_{F_{\{1/2,3\}}},\sigma_{F_{\{1/2,3\}}}^+)$ has the shadowing property, we choose $\delta>0$ so that $\delta<\e/2$ and every $\delta$-pseudo-orbit is $\e/2$-shadowed by a true orbit. Fix $n_0\in\N$ so that $1/2^{n_0}<\delta/2$, and define
	\[
	W=\left\{\xx\in I^+_{F_{\{1/2,3\}}}\colon |x_j-0|<\delta\text{ for all }1\leq j\leq n_0\right\}.
	\]
	Note that $W$ is open and $W\se B(\overline{0},\delta/2)$.
	
	Since $(I^+_{F_{\{1/2,3\}}},\sigma_{F_{\{1/2,3\}}}^+)$ is topologically transitive, there exists $\cc=(c_1,c_2,\ldots)\in B(\aa,\delta/2)$ and $k_0\in\N$ such that $(\sigma_{F_{\{1/2,3\}}}^+)^{k_0}(\cc)\in W$, and there exists $\dd=(d_1,d_2,\ldots)\in B(\overline{0},\delta/2)$ and $l_0\in\N$ such that $(\sigma_{F_{\{1/2,3\}}}^+)^{l_0}(\dd)\in B(\bb,\e/2)$.
	
	Let $n\geq k_0+l_0+1$, and set $m=n-(k_0+l_0+1)+k_0=n-l_0-1$. We construct a sequence $(\xx_j)_{j=1}^\infty$ as follows: Let
	\begin{align*}
		\xx_0&=\left(c_1,c_2,\ldots,c_{k_0+n_0},\frac{1}{2}c_{k_0+n_0},\frac{1}{2^2}c_{k_0+n_0},\frac{1}{2^3}c_{k_0+n_0},\ldots\right),\\
		\xx_j&=(\sigma_{F_{\{1/2,3\}}}^+)^j(\xx_0)\text{ for }1\leq j\leq m,\\
		\xx_j&=(\sigma_{F_{\{1/2,3\}}}^+)^{j-(m+1)}(\dd)\text{ for }j\geq m+1.
	\end{align*}
	First, we observe that $(\xx_j)_{j=1}^\infty$ is a $\delta$-pseudo-orbit. To see this, note from the construction of $\xx_0$ that since $(\sigma_{F_{\{1/2,3\}}}^+)^{k_0}(\cc)\in W$, we must have as well that $(\sigma_{F_{\{1/2,3\}}}^+)^j(\xx_0)\in W$ for all $j\geq k_0$. In particular $\sigma(\xx_m)\in W\se B(\overline{0},\delta/2)$ and $\xx_{m+1}=\dd\in B(\overline{0},\delta/2)$, so by the triangle inequality, $D(\sigma(\xx_m),\xx_{m+1})<\delta$. For all values of $j\in\N\setminus\{m\}$, $\xx_{j+1}=\sigma(\xx_j)$, so this is a $\delta$-pseudo-orbit.
	
	Next, observe that $\xx_0\in B(\aa,\delta/2)\se B(\aa,\e/2)$, and likewise $\xx_n=(\sigma_{F_{\{1/2,3\}}}^+)^{l_0}(\dd)\in B(\bb,\e/2)$. From the shadowing property, there is a point $\yy\in I^+_{F_{\{1/2,3\}}}$ such that $D((\sigma_{F_{\{1/2,3\}}}^+)^j(\yy),\xx_j)<\e/2$ for all $j\geq 0$. In particular, using the triangle inequality once again, we obtain $\yy\in B(\aa,\e)\se U$, and $(\sigma_{F_{\{1/2,3\}}}^+)^n(\yy)\in B(\bb,\e)\se V$. Therefore $(\sigma_{F_{\{1/2,3\}}}^+)^n(U)\cap V\neq\emptyset$. Since this holds for all $n\geq k_0+l_0+1$, we have that $(I^+_{F_{\{1/2,3\}}},\sigma_{F_{\{1/2,3\}}}^+)$ is topologically mixing.
	
	This is a contradiction to Proposition~\ref{no mixing}. Therefore, we must have that $(I^+_{F_{\{1/2,3\}}},\sigma_{F_{\{1/2,3\}}}^+)$ is not shadowing.
\end{proof}
 
 \begin{corollary} The dynamical system $(I_{F_{\{1/2,3\}}}^+,\sigma_{F_{\{1/2,3\}}}^+)$ does not have the specification property.	
 \end{corollary}
 
 \begin{proof}
 Since $\sigma_{F_{\{1/2,3\}}}^+:I_{F_{\{1/2,3\}}}^+ \to I_{F_{\{1/2,3\}}}^+$ is an onto map, this claim is a direct consequence of Theorem \ref{no mixing} and a well known result in \cite{Denker}. 	
 \end{proof}

\subsection{Three line example}

\begin{definition} Define a relation $F_{\{1/2,3,1\}} \subseteq [0,1]^2$ by $$F_{\{1/2,3,1\}}= \left\{ (x,3x):x \in \left[0,\frac{1}{3}\right] \right\} \cup \left\{\left(x,\frac{x}{2}\right): x \in [0,1] \right\} \cup \bigg\{(x,x): x \in [0,1] \bigg\}. $$

\end{definition}

\begin{remark}
	
The particular choice of $(\frac{1}{2},3)$ is not important - any $(r,\rho) \in \mathcal{NC}$ will do. The relation $F_{\{1/2,3,1\}} $ is pictured in Figure \ref{Figure: Three Lines}.

\end{remark} 

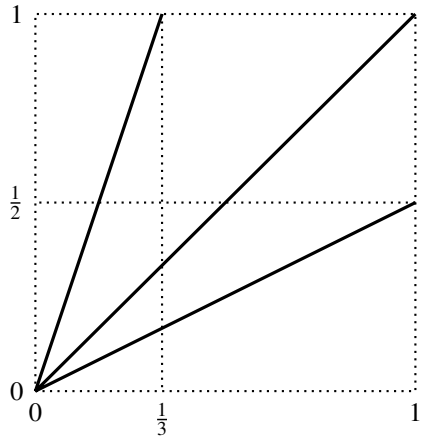
\begin{figure}[h]
	\begin{center}
		\begin{tikzpicture}[scale=5]
			\draw[thick,dotted] (0,0)node[left]{\footnotesize0} -- (0,1)node[left]{\footnotesize1} -- (1,1) -- (1,0)node[below]{\footnotesize1} -- (0,0)node[below]{\footnotesize0};
			\draw[thick, dotted] (1/3,1) -- (1/3,0) node[below]{\footnotesize$\frac{1}{3}$}
			(1,1/2) -- (0,1/2) node[left]{\footnotesize$\frac{1}{2}$};
			\draw[very thick] 
			(0,0) -- (1,1) 
			(0,0) -- (1,1/2)
			(0,0) -- (1/3,1);
		\end{tikzpicture}
		\caption{Relation $F_{\{1/2,3,1\}}$}\label{Figure: Three Lines}
	\end{center}
\end{figure}

It was shown in \cite{BE2} that $(I_{F_{\{1/2,3,1\}}}^+,\sigma_{F_{\{1/2,3,1\}}}^+)$ is a  topologically mixing dynamical system. We show below that  it has the specification property, but does not have the shadowing property. The lack of shadowing in this closed relation is due to the fact that it contains the identity and the identity is ``separated'' from the rest of the relation.

Let $[a,b]$ denote a closed interval in $[0,1]=I$, and for convenience, if $c \ge 0$, $c[a,b]=[ca,cb]$.

Let $\mathcal{A}$ be the collection of all finite subsets $\{ a_1, a_2, a_3, \dots , a_k \} $ of the positive real numbers such that $k \ge 3 $ , $a_1=3$, $a_2=1$, $a_3 =\frac{1}{2}$, and such that $\frac {1}{3} \leq a_i \leq 3$ for each $i\in \{1,2,3, \ldots ,k \} $. Note that every set in $\mathcal{A}$ is an LF-inducing set. 

For each  $A \in \mathcal{A}$, let $F_A = \{ (x,y) \in I ^2 : y= a x  \text{ for some } a \in A \}$. Let $\mathcal{F}= \{ F_A : A\in \mathcal{A} \}$.

\begin{observation} If $G \in \mathcal{F}$ and $B \subset [0,1]$,  then $B \subset G(B) .$
\end{observation}

\begin{observation} \label{chrisvanobs1}
	If $G\in \mathcal{F}$ and $(x,y)\in G$ then $\frac{1}{3}x \le y \le 3x $
\end{observation} 

A map $f:X\to  X$ is \emph{continuum-wise exact} if for every $\delta>0$ there exists $N=N_{\delta}$ such that if $K$ is a subcontinuum of $X$ and $\mbox{diam}(K)\geq \delta$, then $f^N(K)=X$.

It is not hard to show that continuum-wise exact maps have the specification property, and while $F_{\{1/2,3,1\}}$ is not continuum-wise exact, it almost is--except for what happens at $0$.   The proof to show that $F_{\{1/2,3,1\}}$ has the specification property is motivated by the proof that continuum-wise exact maps have the specification property with additional care taken when points are close to $0$. (For the interested reader, a map of the pseudo-arc with a continuum-wise exact map can be found in \cite{chris}.)

Let $y_0,y_1, \dots, y_p$ be some finite trajectory of $y_0$ under $F$. Then we can define 

$$y_i=3^m (1)^k \frac{1}{2^n} y_0 = \frac{3^m}{2^n} y_0 := \frac{3^{m_i}}{2^{n_i}} y_0 ,$$ where $m+k+n=i$, or equivalently, $m+n \le i$ ($m_i+n_i \le i$).

Let $y_0,y_1, \dots, y_p$ be some finite trajectory of $y_0$ under $F$. Then we can define 

$$y_i=3^m (1)^k \frac{1}{2^n} y_0 = \frac{3^m}{2^n} y_0 := \frac{3^{m_i}}{2^{n_i}} y_0 ,$$ where $m+k+n=i$, or equivalently, $m+n \le i$ ($m_i+n_i \le i$).

\begin{lemma}\label{altlemma}
	Assume  $\delta $ and $\gamma$ are elements of $(0,1)$.  Then there exists $N_{\delta,\gamma} \in \mathbb{N} $ such that if $[a,b] \subset (0,1]$ and $b-a > \delta$, then $[\gamma,1] \subset F_{\{1/2,3,1\}}^n([a,b]) $ for each $n \geq N_{\delta,\gamma} $.
\end{lemma}

	\begin{proof}
		Let $k \in \mathbb{N}$ such that $\frac{1}{k} < \frac{\delta}{2}$ and let $B_i=[\frac{i}{k},\frac{i+1}{k}]$ for each $i \in \{0,1, \ldots,k-1 \}$.	There is $i_0 \in \{0,1, \ldots, k-1 \}$ such that $B_{i_0} \subset [a,b]$. Since $\{ \frac{3^p}{2^q}: p,q \in \mathbb{N}_0 \}$ is dense in $\mathbb{R}$, there are $p,q \in \mathbb{N}_0$ such that $\frac{3^p}{2^q} (\frac{i_0}{k}) <1< \frac{3^p}{2^q} (\frac{i_0+1}{k})$.
		\vphantom{}
		
		Let $a_{\delta}= \frac{3^p}{2^q} (\frac{i_0}{k})$ and $N_{\delta}=p+q$. Then if $n \ge N_{\delta}=p+q$, $[\frac{3^p}{2^q} (\frac{i_0}{k}),1] \subset F_{\{1/2,3,1\}}^{p+q}([a,b])$. Then $ [a_{\delta},1]=[\frac{3^p}{2^q} (\frac{i_0}{k}),1] \subset F_{\{1/2,3,1\}}^{p+q}([a,b]) =F_{\{1/2,3,1\}}^{N_{\delta}}([a,b]) \subset F_{\{1/2,3,1\}}^n([a,b])  $.
		
		There is $c_0=\frac{3^{m_0}}{2^{n_0}}$ such that $a_{\delta} < c_0 < \frac{a_{\delta}+1}{2}$. (Note that $\frac{a_{\delta}+1}{2}<1$.) Similarly, there is $c_1=\frac{3^{m_1}}{2^{n_1}}$ such that $c_0 a_{\delta} < c_1 < a_{\delta}$. Continuing inductively, suppose $c_0,c_1, \ldots, c_k$ have been found. Then there is $c_{k+1}=\frac{3^{m_{k+1}}}{2^{n_{k+1}}}$ such that $c_0c_1 \cdots c_k a_{\delta} < c_{k+1} < c_0 c_1 \cdots c_{k-1} a_{\delta}$.  For each positive integer $k$, let $N_k = n_k + m_k$. Note that 
		
		\begin{itemize}
			\item $F_{\{1/2,3,1\}}^{N_0}([a_{\delta},1]) \supset c_0[a_{\delta},1] \cup [a_{\delta},1]=[c_0 a_{\delta},1] $;
			\item $F_{\{1/2,3,1\}}^{N_0+N_1}([a_{\delta},1]) \supset c_1[c_0 a_{\delta},1] \cup [c_0 a_{\delta},1] =[c_0c_1a_{\delta},1]$; and - continuing until the $k$th step -
			\item  $F_{\{1/2,3,1\}}^{N_0+N_1+ \cdots+N_k}([a_{\delta},1]) \supset c_k[c_0c_1 \cdots c_{k-1} a_{\delta},1] \cup [c_0c_1 \cdots c_{k-1}a_{\delta},1] = [c_0 \cdots c_k a_{\delta},1]$.
			
		\end{itemize} 
		
		Now $\lim_{k \to \infty} c_0c_1 \cdots c_k a_{\delta}=0$ so there is $j \in \mathbb{N}$ such that $c_0c_1 \cdots c_j a_{\delta} < \gamma $. Let $N_{\delta,\gamma}= N_{\delta}+ \Sigma_{i=o}^j N_i$. Then $F_{\{1/2,3,1\}}^n([a,b]) \supset [\gamma,1]$ for each $n \ge N_{\delta,\gamma}$. 
		
	\end{proof}

\begin{definition}
	Suppose $p$ is a positive integer and for each $0 \le i \le p$, $[a_i,b_i]$ is an interval contained in $[0,1]$. Then $([a_i,b_i])_{i=0}^p$ is a \emph{descending interval trajectory} of $G \in \mathcal{F}$  on $[a_0,b_0]$ if $[a_{i+1},b_{i+1}] \subset G([a_i,b_i])$ for all $0 \le i <p$ 
\end{definition}

\begin{lemma} \label{chrisvanlemma1}
	Let $G\in \mathcal{F}$ and let $(y_i)_{i=0}^p =(c_iy_0)_{i=0}^p$ (with $c_0=1$) be some finite trajectory of $y_0$ under $G$ . Let $1 \ge \epsilon >0$, $p \in \mathbb{N}$, and $y_0 > \frac{2 \epsilon}{9}$. Suppose $y_p \ge \frac{\epsilon}{3}$. Then if $a_0 \in (0,1]$ such that $y_0 - \epsilon \le a_0 \le y_0-\frac{2 \epsilon}{9}$, there is a descending interval trajectory  $([a_i,b_i])_{i=0}^p =([a_i,y_i])_{i=0}^p$  under $G$ where $b_i=y_i$  such that 
	\begin{enumerate}
		\item $[a_i,y_i] \subset \overline{B}_{\epsilon}(y_i)=[y_i - \epsilon,y_i+\epsilon]$ for each $i \in \{0,1, \ldots, p\}$, and 
		\item $\diam([a_p,y_p]) \ge \frac{\epsilon^2}{9}$. 	
	\end{enumerate}
	
\end{lemma}

\begin{proof}
	Assume the hypotheses. Since $y_0 > \frac{2 \epsilon}{9}$, there is $a_0 \in (0,1]$ such that $y_0 - \epsilon \le a_0 \le y_0 - \frac{2 \epsilon}{9}$ (and any such $a_0$ will do). Then $[a_0,y_0] \subset [y_0 - \epsilon, y_0 + \epsilon]$. 
	
	\vphantom{}
	
	Now $a_0>0$ and $\frac{c_1}{c_0} a_0 = c_1a_0 >0$. Consider 
	$$ [\frac{c_1}{c_0} a_0, y_1] \cap [y_1 - \epsilon,y_1 + \epsilon] = \frac{c_1}{c_0}[a_0,y_0] \cap [y_1 - \epsilon,y_1].$$ 
	Let $a_1= \max \{c_1a_0,y_1-\epsilon\}$. Then $[a_1,y_1] \subset [y_1 - \epsilon,y_1+\epsilon]$ and $G([a_0,y_0]) \supset [a_1,y_1]$. 
	Continuing inductively, assume $[a_i,y_i]$ ($i <p$) was chosen so that $0 < a_i \le 1$, $[a_i,y_i] \subset [y_i - \epsilon,y_i+\epsilon]$ and $G([a_{i-1},y_{i-1}]) \supset [a_i,y_i]$  . Then consider 
	$$ \frac{c_{i+1}}{c_i}[a_i,y_i] \cap [y_{i+1} - \epsilon, y_{i+1} + \epsilon]=[\frac{c_{i+1}}{c_i}a_i,y_{i+1}] \cap [y_{i+1} - \epsilon, y_{i+1}].$$
	
	Let $a_{i+1}= \max \{\frac{c_{i+1}}{c_i}a_i, y_{i+1} -\epsilon\}$. Then $[a_{i+1},y_{i+1}] \subset [y_{i+1} - \epsilon,y_{i+1}+\epsilon]$ and $G([a_i,y_i]) \supset [a_{i+1},y_{i+1}]$ . 
	
	\vphantom{}
	
	Let $k = \max \{i \in \{0,1, \ldots, p \}: \diam([a_i,y_i]) \ge \frac{ \epsilon}{3} \}$. If $k=p$, then $\diam([a_p,y_p]) \ge \frac{ \epsilon}{3}\geq \frac{\epsilon ^2}{9}$, and we are done. So suppose $k < p$. Then $\diam([a_i,y_i]) < \frac{ \epsilon}{3}$ for $i>k$. Also, this means that for $i>k$, $[a_i,y_i] = \frac{c_i}{c_k} [a_k,y_k]$, for otherwise $[a_i,y_i]=[y_i-\epsilon,y_i]$ and $\diam([a_i,y_i]) = \epsilon > \frac{\epsilon}{3}$. Moreover, $$\frac{\epsilon}{3} \le y_p = \frac{c_p}{c_k}y_k \le \frac{c_p}{c_k}.$$ Then $\diam([a_p,y_p])=\frac{c_p}{c_k}\diam([a_k,y_k] \ge \frac{ \epsilon}{3} (\frac{\epsilon}{3})= \frac{\epsilon^2}{9}$. 
	
\end{proof}

\begin{lemma} \label{chrisvanlemma2}
	Let $G\in \mathcal{F}$, let $1 \ge \epsilon>0$, and let $(y_i)_{i=0}^p =(c_iy_0)_{i=0}^p$ (with $c_0=1$) be some finite trajectory of $y_0$ under $G$ . Suppose $y_i < \epsilon$ for all $i \in \{0,1,\ldots,p\}$. Let $[a_i,b_i]=[\frac{\epsilon}{9}, \epsilon]$ for $0 \le i \le p$. Then $([a_i,b_i])_{i=0}^p$ is a descending interval trajectory under $G$ such that 
	\begin{enumerate}
		\item $[a_i,b_i] \subset \overline{B}_{\epsilon}(y_i)$ for $0 \le i \le p$, and
		\item $\diam([a_p,b_p]) \ge \frac{\epsilon^2}{9}$.
	\end{enumerate} 
	
\end{lemma}

\begin{proof}
	This follows from the fact that $$y_i - \epsilon \le 0 <a_i <b_i \le \epsilon \le y_i + \epsilon$$ for each $0 \le i \le p$. 	
\end{proof}

\begin{lemma} \label{chrisvanlemma3}
	Let $G \in \mathcal{F}$, and let $1 \ge \epsilon >0$ and let $(y_i)_{i=0}^p =(c_iy_0)_{i=0}^p$ (with $c_0=1$) be some finite trajectory of $y_0$ under $G$. 	Let $y_0 > \frac{2 \epsilon}{9}$ and $y_p < \frac{\epsilon}{3}$. Then there exists a descending interval trajectory $([a_i,b_i])_{i=0}^p =([a_i,y_i])_{i=0}^p  $ under $G$ such that
	\begin{enumerate}
		\item $[a_i,b_i] \subset \overline{B}_{\epsilon}(y_i)$ for $0 \le i \le p$, and
		\item $\diam([a_p,b_p]) \ge \frac{\epsilon^2}{9}$.
	\end{enumerate} 
\end{lemma}

\begin{proof}
	Now $y_p < \frac{\epsilon}{3}$ and $y_0 > \frac{2 \epsilon}{9}$, and there is $a_0 \in (0,1]$ such that $a_0 \le y_0 - \frac{2 \epsilon}{9}$. If $y_i < \epsilon$ for $0 \le i \le p$, then by Lemma \ref{chrisvanlemma2}, the claim holds. So suppose there is $i \in \{0,1,\ldots,p \}$ such that $y_i \ge \epsilon$. Let $k= \max \{i: 0 \le i \le p, y_i \ge \frac{\epsilon}{3} \}$. Then $k < p$ and $y_{k+1} < \frac{\epsilon}{3}$  . Also, according to Observation \ref{chrisvanobs1},  $y_{k+1} \geq \frac{y_k}{3}$ . So $y_k <\epsilon $. 
	
	\vphantom{}
	
	Since we have $y_0 - \epsilon \le a_0 \le y_0-\frac{2 \epsilon}{9}$, by Lemma \ref{chrisvanlemma1}, there is a descending interval trajectory under $G$  $([a_i,y_i])_{i=0}^k = ([a_i,b_i])_{i=0}^k$ (where $b_i = y_i$) such that 
	\begin{enumerate}
		\item $[a_i, y_i] \subset [y_i - \epsilon, y_i + \epsilon]$ for $0 \le i \le k$, and 
		\item $\diam([a_k,b_k] \ge \frac{\epsilon^2}{9}$.	
	\end{enumerate}
	
	For $k+1 \le i \le p$, define $[a_i,b_i] = [a_k,b_k] \subset [0,\epsilon] \subset \overline{B}_{\epsilon}(y_i)$. Now, since $[a_p,b_p]=[a_k,b_k]$, we have $\diam([a_p,b_p]) \ge \frac{\epsilon^2}{9}$.

\end{proof}

\begin{lemma} \label{chrisvanlemma4}
	Let $G \in \mathcal{F} $, let $1 \ge \epsilon>0$, and let $(y_i)_{i=0}^p =(c_iy_0)_{i=0}^p$ (with $c_0=1$) be some finite trajectory of $y_0$ under $G$ .  Suppose $y_0 < \epsilon$. Then there exists a descending interval trajectory   $([a_i,b_i])_{i=0}^p$ under $G$ such that 
	\begin{enumerate}
		\item $[a_0,b_0] \subset [\frac{\epsilon}{9},1]$,
		\item $[a_i,b_i] \subset \overline{B}_{\epsilon}(y_i)$ for $0 \le i \le p$, and 
		\item $\diam([a_p,b_p]) \ge \frac{\epsilon^2}{9}$.
	\end{enumerate}
	
\end{lemma}

\begin{proof} If $y_i < \epsilon$ for $0 \le i \le p$, then by Lemma \ref{chrisvanlemma2}, the claim holds. So suppose that there is $i_0 \in \{0,1,\ldots,p \}$ such that $y_{i_0} \ge \epsilon$. Let $k = \min \{i: 0 \le i \le p, y_i \ge \epsilon\}$. (Note that $k$ exists since $y_{i_0} \ge \epsilon$. Also, $k > 0$.)
	
	\vphantom{}
	
	Since $y_k \ge \epsilon$ and $y_{k-1}< \epsilon$, and  by Observation \ref{chrisvanobs1}, $y_k \leq 3y_{k-1}$, it follows that $\frac{\epsilon}{3} \le y_{k-1} < \epsilon$. For $i \in \{0,1, \ldots,k-1\}$ define $[a_i,b_i]= [\frac{\epsilon}{9}, y_{k-1}] \subset [\frac{\epsilon}{9}, 1]$. Then $[a_0,b_0]=[\frac{\epsilon}{9}, y_{k-1}] \subset [\frac{\epsilon}{9}, 1]$, so condition (1) is satisfied. Since for each $i \in \{0,1, \ldots, k-1 \}$, $\overline{B}_{\epsilon}(y_i) \cap [0,1]= [0,y_i + \epsilon]$, $[a_i,b_i] = [\frac{\epsilon}{9}, y_{k-1}] \subset [\frac{\epsilon}{9}, \epsilon] \subset [0, y_i + \epsilon] \subset \overline{B}_{\epsilon}(y_i)$. Hence, condition (2) is satisfied (up to $k-1$). 
	
	\vphantom{}
	
	Furthermore, $\diam([a_{k-1},b_{k-1}] = \diam([\frac{\epsilon}{9}, y_{k-1}] \ge \frac{\epsilon}{3} - \frac{\epsilon}{9}= \frac{2 \epsilon}{9}$, so $y_{k-1} - \epsilon < 0 \le a_{k-1} = \frac{\epsilon}{9} < y_{k-1} - \frac{2 \epsilon}{9}$. 
	
	\vphantom{}
	
	In order to construct  $[a_i,b_i]$ for $i \in \{k, \ldots, p\}$ note that either $y_p \ge \frac{\epsilon}{3}$ or $y_p < \frac{\epsilon}{3}$. 
	\begin{itemize}
		\item If $y_p \ge \frac{\epsilon}{3}$, apply Lemma \ref{chrisvanlemma1} to $(y_i)_{i=k-1}^p = (\frac{c_i}{c_{k-1}} y_{k-1})_{i=k-1}^p$ to obtain a descending interval trajectory under $G$ $([a_i,b_i])_{i=k-1}^p$ such that  (1) for $k-1 \le i \le p$, $[a_i,y_i]=[a_i,b_i] \subset \overline{B}_{\epsilon}(y_i)$, and (2) $\diam ([a_p,y_p]) \ge \frac{\epsilon^2}{9}$.
		\item If $y_p < \frac{\epsilon}{2}$, apply Lemma \ref{chrisvanlemma3} to $(y_i)_{i=k-1}^p = (\frac{c_i}{c_{k-1}} y_{k-1})_{i=k-1}^p$ to obtain a descending interval trajectory under $G$  $([a_i,b_i])_{i=k-1}^p = ([a_i, y_i])_{i=k-1}^p $ such that (1)  for $k-1 \le i \le p$, $[a_i,y_i]=[a_i,b_i] \subset \overline{B}_{\epsilon}(y_i)$, and (2) $\diam ([a_p,y_p]) \ge \frac{\epsilon^2}{9}$.
		
	\end{itemize}

\end{proof}

\begin{theorem} \label{chrisvanthm1}
	Let $ G\in \mathcal{F}$, let $1 \ge \epsilon >0$, and let $(y_i)_{i=0}^p =(c_iy_0)_{i=0}^p$ (with $c_0=1$) be some finite trajectory of $y_0$ under $G$. Then there exists a descending interval trajectory  $([a_i,b_i])_{i=0}^p$ under $G$  such that 
	\begin{enumerate}
		\item $[a_0,b_0] \subset [\frac{\epsilon}{9},1]$,
		\item $[a_i,b_i] \subset \overline{B}_{\epsilon}(y_i)$ for $0 \le i \le p$, and 
		\item $\diam([a_p,b_p]) \ge \frac{\epsilon^2}{9}$.
	\end{enumerate}

\end{theorem}

\begin{proof} Suppose $G \in \mathcal{F}$ ,  $1 \ge \epsilon >0$, and $(y_i)_{i=0}^p =(c_iy_0)_{i=0}^p$ (with $c_0=1$) is some finite trajectory of $y_0$ under $G$ .
	\begin{itemize}
		\item \textit{Case 1} If $y_0 < \epsilon$, then the result follows from Lemma \ref{chrisvanlemma4}.
		\item \textit{Case 2} If $y_0 \ge \epsilon$ and $y_p < \frac{\epsilon}{3}$, then Lemma \ref{chrisvanlemma3} applies and conditions (2) and (3) are satisfied. Furthermore, $y_0 \ge \epsilon$ implies $y_0 - \epsilon \ge 0$ and $y_0 - \frac{2 \epsilon}{9} \ge \frac{\epsilon}{9}$. Then condition (1) is also satisfied, because $a_0$ can be chosen to be $y_0 - \frac{2 \epsilon}{9}$. Also, $b_0=y_0$, so $[a_0,b_0]=[y_0 - \frac{2 \epsilon}{9},y_0]$ and $a_0=y_0 - \frac{2 \epsilon}{9} \ge \frac{\epsilon}{9}$. 
		\item \textit{Case 3} If $y_0 \ge \epsilon$ and $y_p \ge \frac{\epsilon}{3}$, then Lemma \ref{chrisvanlemma1} applies (4.13) and conditions (2) and (3) are satisfied. Again, $y_0 -\frac{2 \epsilon}{9} \ge \frac{\epsilon}{9}$, and $a_0$ can be chosen to be $y_0 - \frac{2 \epsilon}{9}$. Also, $y_0=b_0 \in (0,1]$, so $[a_0,b_0]= [y_0 - \frac{2 \epsilon}{9},y_0] \subset [\frac{\epsilon}{9},1]$. Thus, condition (1) is satisfied.
	\end{itemize}

	Since each trajectory $(y_i)_{i=0}^p$ falls into exactly one of cases (1),(2),(3), the theorem holds.
	
\end{proof}

\begin{observation}
	In the proof of the following theorem we assume that $\epsilon \leq  1$. It doesn't weaken our results since for $\epsilon >  1$ we trivially get the CR-specification property.
\end{observation}

\begin{theorem} \label{G has the CR-specification property}
	Every  $G \in \mathcal{F}$ has the CR-specification property.	
\end{theorem}

\begin{proof} 
		
		Assume $G \in \mathcal{F}$.  Let $1 \ge \epsilon >0$, $\delta=\frac{\epsilon^2}{9}$,  and $\gamma = \frac{\epsilon}{18}$. Then by Lemma \ref{altlemma}, there is $N_{\delta,\gamma} \in \mathbb{N}$ such that if $[a,b]$ is an interval in $(0,1]$ with $\diam([a,b]) \ge \delta$, then $[\gamma,1] \subset F_{\{1/2,3,1\}}^n([a,b]) \subset G^n([a,b])$ for all $n \ge N_{\delta,\gamma}$.

	Suppose 
	
	$$ S = (\pi_{[k_1,l_1]}(\mathbf{x}_1),\pi_{[k_2,l_2]}(\mathbf{x}_2), \ldots, \pi_{[k_m,l_m]}(\mathbf{x}_m))$$ is an $N_{\delta,\gamma}$-spaced CR specification in $G$. We need to show that there is $\mathbf{y} \in I_G^+$ that $\epsilon$-traces $S$:
	
	Now each $\pi_{[k_i,l_i]}(\mathbf{x}_i)$ for $1 \le i \le m$ is a finite trajectory (of $\pi_{k_i}(\mathbf{x}_i)$). Then, applying Theorem \ref{chrisvanthm1}, there exists for each $1 \le i \le m$ a descending interval trajectory  under $G$ $$([a_{k_i}+j,b_{k_i}+j])_{j=0}^{l_i-k_i}$$  such that 
	
	\begin{enumerate}
		\item $[a_{k_i},b_{k_i}] \subset [\frac{\epsilon}{9},1]$,
		\item $[a_{k_i}+j,b_{k_i}+j] \subset \overline{B}_{\epsilon}(\pi_{{k_i+j}}(\mathbf{x}_i))$ for $0 \le j \le l_i-k_j$, and 
		\item $\diam([a_{l_i},b_{l_i}]) \ge \frac{\epsilon^2}{9}$.
	\end{enumerate}
	
	Because  $([a_{k_m}+j,b_{k_m}+j])_{j=0}^{l_m-k_m}$ is a descending interval trajectory under $G$ there is for each $y \in [a_{l_m},b_{l_m}]$ an orbit segment $\{y_{k_m}, y_{k_m +1} , y_{k_m +2}, \ldots , y_{k_{l_m}} \}$ under $G$ such that  $y=y_{k_l}$ and $y_{k_m}+j \in [a_{k_i}+j , b_{k_i}+j]$  for each  $0 \le j \le l_i-k_j$. Choose and fix $y_{k_l} $.

	Since $\diam([a_{l_{m-1}},b_{l_{m-1}}]) \ge \frac{\epsilon^2}{9}=\delta$ and $k_m - l_{m-1} \ge N_{\delta,\gamma}$  there is, according to the Lemma \ref{altlemma}, an orbit cycle $\{ y_{l_{m-1}} , y_{l_{m-1} +1} , y_{k_{m-1} + 2 }, \ldots , y_{k_{m}} \}$ under $G$ such that $y_{l_{m-1}} \in [a_{l_{m-1}},b_{k_{m-1}}] $ .

	Repeating this construction for each orbit segment in the specification $S$ we obtain an orbit segment under $G$ ,  
	$\{ y_{k_1} , y_{{k_1}+ 1} , y_{{k_1}+ 2} , \ldots , y_{k_{l_m}}  \}$ . If $k_1 > 0$ let $y(i)= y_{k_1}$ for each $ 0 \le i \le k_1$ and let $y_i = y_{k_{l_m}}$ for each $i \ge k_{l_m}$.  Then $(y_i) _{i=0}^\infty $ is an element of $I_G^+$ that $\epsilon$-traces $S$.

\end{proof}

\begin{theorem}\label{Theorem: three line has specification}
	For each $G \in \mathcal{F}$ both $I_G^+$ and $I_G$ are homeomorphic to the Lelek fan and both $(I_G^+,\sigma_G^+)$ and $(I_G,\sigma_G)$ have the specification property.
\end{theorem}

\begin{proof}
	Assume $G \in \mathcal{F}$. By Theorem \ref{Proposition: LF-inducing yields Lelek fan (one-sided)}, $I_G^+$ is a Lelek fan. By Theorem \ref{Proposition: LF-inducing yields Lelek fan (two-sided)}, $I_G$ is a Lelek fan. By Theorem \ref{CR-spec}, and Theorem \ref{G has the CR-specification property}, $(I_G^+,\sigma_G^+)$ has the specification property. By \cite[Theorem 3.19]{BEJK}, $(I_G^+,\sigma_G^+)$ has the specification property if and only if $(I_G,\sigma_G)$ has the specification property.
\end{proof}

\begin{theorem} The dynamical system $(I_{F_{\{1/2,3,1\}}}^+, \sigma_{F_{\{1/2,3,1\}}}^+)$ does not have the shadowing property.
	
\end{theorem}

\begin{proof}
	Let $0<\e<1/12$, and let $\delta>0$ be arbitrary. Fix $n_0\in\N$ such that $1/(4n_0)<\delta$. Given $x\in[0,1]$, denote by $\overline{x}$ the sequence $(x,x,x,\ldots)\in I^+_{F_{\{1/2,3,1\}}}$. 
	
	Define a sequence $(\xx_k)_{k=0}^\infty$ as follows: for $0\leq k\leq n_0$, let $\xx_k=\overline{\frac{1}{2}+\frac{k}{2n_0}}$, and for $k>n_0$, let $\xx_k=\overline{1}$. Observe that this is a $\delta$-pseudo-orbit as when $0\leq k<n_0$,
	\[
	D\left[\sigma_{F_{\{1/2,3,1\}}}^+\left(\xx_k\right),\xx_{k+1}\right]=D\left[\xx_k,\xx_{k+1}\right]=\frac{1}{4n_0}<\delta,
	\]
	and for $k\geq n_0$, $\sigma_{F_{\{1/2,3,1\}}}^+(\xx_k)=\xx_{k+1}.$
	
	We show this $\delta$-pseudo-orbit cannot be $\e$-shadowed by any true orbit. Let $\zz\in I^+_{F_{\{1/2,3\}}}$ and suppose $D((\sigma_{F_{\{1/2,3\}}}^+)^n(\zz),\xx_n)<\e$ for all $n\geq0$. Since $\e<1/8$, we have
	\[
	\frac{1}{12}>D(\zz,\xx_0)=\sup_{j\in\N}\left\{\frac{\left|z_j-\frac{1}{2}\right|}{2^j}\right\},
	\]
	so for all $j\in\N$, $|z_j-1/2|<2^j/12$. In particular, $|z_1-1/2|<1/6$, so $z_1>1/3$. This means that $z_2\neq 3z_1$, so $z_2\leq z_1$.
	
	Likewise, since $D(\sigma_{F_{\{1/2,3,1\}}}^+(\zz),\xx_1)<\e$, we have that 
	\[
	\left|z_2-\left(\frac{1}{2}+\frac{1}{2n_0}\right)\right|<\frac{1}{6}.
	\]
	This again yields  $z_2>1/3$, and $z_3\leq z_2$. Continuing inductively, we have that $\zz$ is a non-increasing sequence. This is a contradiction since $|z_1-1/2|<2\e$ while for $j\geq n_0$, $|z_j-1|<2\e$.
	
	Therefore $(I^+_{F_{\{1/2,3,1\}}},\sigma_{F_{\{1/2,3,1\}}}^+)$ does not have the shadowing property.
\end{proof}

\subsection{$n$-lines examples} 

 What happens if we keep adding line segments to our relation? Previously we showed that the induced Mahavier products are still Lelek fans. In this subsection and the next section, we show that  (1) if we include a line segment $ax$ and its inverse $\frac{x}{a}$,  the associated shift maps have a dense set of periodic points, and (2) if the relation $F$ (consisting of line segments in $[0,1]^2$ containing the origin $(0,0)$) has the property that for some positive integer $M$, the diagonal is a subset of $F^n$ for all positive integers $n \ge M$, then the shift maps are topologically mixing.
 \\
 The following theorem is a generalization of \cite[Theorem 3.13]{BE2}.
 \begin{theorem} \label{Proposition: Eventual diagonal and transitive implies mixing}
 	Let $(X,F)$ be a CR-dynamical system. If $(X_F,\sigma_F)$ is topologically transitive and there exists a positive integer $M$ such that the diagonal $\Delta_X \subseteq F^n$ for all $n \ge M$, then $(X_F,\sigma_F)$ is topologically mixing. 
 	
 \end{theorem}
 
 \begin{proof}
 	Suppose $(X_F,\sigma_F)$ is topologically transitive and there exists a positive integer $M$ such that $\Delta_X\se F^n$ for all $n\geq M$. Let $U,V\se X_F$ be non-empty open sets. There exists a positive integer $N$ and non-empty open sets $\tilde{U},\tilde{V}\se\bigstar_{i=-N}^{N-1} F$ such that $\pi^{-1}_{\{-N,\ldots,N\}}(\tilde{U})\se U$ and $\pi^{-1}_{\{-N,\ldots,N\}}(\tilde{V})\se V$. Since $(X_F,\sigma_F)$ is topologically transitive, there exists an integer $k\geq 1$ and $\mathbf{x}\in \pi^{-1}_{\{-N,\ldots,N\}}(\tilde{U})$ such that $\sigma_F^k(\mathbf{x})\in \pi^{-1}_{\{-N,\ldots,N\}}(\tilde{V})$. Without loss of generality, we may suppose $k\geq 2N$.
 	
 	Let $n\geq M+k$. Then $n-k\geq M$, so $\Delta_X\se F^{n-k}$. Then there is a finite sequence $(y_0,y_1,\ldots,y_{n-k})\in\bigstar_{i=0}^{n-k-1}F$ such that $y_{n-k}=y_0=x_N$. We may thus define a point $\mathbf{z}\in X_F$ by
 	\[
 	\mathbf{z}=\left(\ldots,x_{-N},\ldots,x_N,y_1,y_2,\ldots,y_{n-k},x_{N+1},x_{N+2},\ldots\right).
 	\]
 	Then we have $z_i=x_i$ for all $i\in\{-N,\ldots,N\}$ and $z_{i+(n-k)}=x_{i}$ for all $i\in\{-N+n,\ldots,N+n\}$, so we have 
 	\[
 	\mathbf{z}\in \pi^{-1}_{\{-N,\ldots,N\}}(\tilde{U})\se U,
 	\]
 	and
 	\[
 	\sigma^n\left(\mathbf{z}\right)=\sigma^k\left(\sigma^{n-k}\left(\mathbf{z}\right)\right)\in\pi^{-1}_{\{-N,\ldots,N\}}(\tilde{V})\se V.
 	\]
 	
 	Since this holds for all $n\geq M+k$, we get that $(X_F,\sigma_F)$ is topologically mixing.
 \end{proof}
 
 \begin{observation}
Since  	$(X_F,\sigma_F)$ is semi-conjugate to $(X_F^+,\sigma_F^+)$, previous theorem also hold for $(X_F^+,\sigma_F^+)$.
 \end{observation}
 
 \begin{theorem}\label{Theorem: F union F^{-1} has dense periodic}
 	Let $X$ be a compact metric space and let $F \subseteq X \times X$ be a non-empty closed relation. If $G=F \cup F^{-1}$, then the Mahavier dynamical system $(X_G, \sigma_G)$ has a dense set of periodic points. 
 \end{theorem}
 
 \begin{proof}
 Since $G=F\cup F^{-1}$, for all $a,b\in X$, it holds that $(a,b)\in G$ if and only if $(b,a)\in G$.	Let $\mathbf{x}\in X_G$. Choose $N\in\N$ such that $2^{-(N+1)}<\e$. Let $\mathbf{y}$ be the finite sequence $\mathbf{y}=(x_{-N},x_{-(N-1)},\ldots,x_{N-1})$, and let $\mathbf{\hat{y}}$ be the finite sequence $\mathbf{\hat{y}}=(x_N,x_{N-1},\ldots,x_{-(N-1)}$. We define the infinite sequence $\mathbf{z}$ by concatenating copies of $\mathbf{y}$ and $\mathbf{\hat{y}}$ in alternating order:
 \[
 \mathbf{z}=\cdots\mathbf{y}\mathbf{\hat{y}}\mathbf{y}\mathbf{\hat{y}}\mathbf{y}\mathbf{\hat{y}}\cdots
 \]
 with $z_i=x_i$ for all $i\in\{-N,\ldots,N\}$.
 
 For all $i\in\Z$, either $(z_i,z_{i+1})$ or $(z_{i+1},z_i)$ appears within the sequence $(x_{-N},\ldots,x_{N})$, so we have $(z_i,z_{i+1})\in G$. Thus $\mathbf{z}\in X_G$. From the construction, $\sigma_G^{2N}(\mathbf{z})=\mathbf{z}$, so it is periodic, and
 \[
 D\left(\mathbf{z},\mathbf{x}\right)<\frac{1}{2^{N+1}}<\e.
 \]
 Therefore $(X_G,\sigma_G)$ has a dense set of periodic points.
 \end{proof}

\begin{lemma}\label{Lemma: Open sets}
	Let $\Omega$ and $\Lambda$ be LF-inducing sets with $\Lambda\se\Omega$. If $U\se\bigstar_{i=1}^n F_{\Omega}$ is a non-empty  set that is open in $\bigstar_{i=1}^n F_{\Omega}$, then the projections
	\[
	\pi_{\{n+1,n+2\}}\left(U\star F_\Lambda\right)~~\text{ and }~~\pi_{\{0,1\}}\left(F_\Lambda\star U\right)
	\]
	contain non-empty sets that are open in $F_\Lambda$.
\end{lemma}
\begin{proof}
	Let $U\se\bigstar_{i=1}^n F_{\Omega}$ be a non-empty  set, open in $\bigstar_{i=1}^n F_{\Omega}$. Since $F_{\Omega}$ contains no vertical or horizontal lines, we have $\pi_{n+1}(U)$ is a non-degenerate (not necessarily open) interval. Thus $\pi_{n+1}(U)$ contains an open interval $(a,b)$. This implies that $\pi_{\{n+1,n+2\}}(U\star F_\Lambda)$ contains a non-empty open set.
\end{proof}

\begin{prop}\label{Proposition: LF-transitive/mixing inherited from sub-relations}
	Let $\Omega$ and $\Lambda$ be LF-inducing sets, and suppose $\Lambda\se\Omega$. 
	\begin{enumerate}
		\item\label{LF-transitive/mixing 1} If $(I_{F_\Lambda},\sigma_{F_\Lambda})$ is topologically transitive, then $(I_{F_\Omega},\sigma_{F_\Omega})$ is topologically transitive.
		\item\label{LF-transitive/mixing 2} If $(I_{F_\Lambda},\sigma_{F_\Lambda})$ is topologically mixing, then $(I_{F_\Omega},\sigma_{F_\Omega})$ is topologically mixing.
	\end{enumerate}
	
\end{prop}
\begin{proof}
	We begin by proving \eqref{LF-transitive/mixing 1}. Suppose $(I_{F_\Lambda},\sigma_{F_\Lambda})$ is topologically transitive, and let $U,V\se I_{F_{\Omega}}$ be non-empty open sets. Then there exists $N\in\N$ and non-empty open sets $\tilde{U},\tilde{V}\se\bigstar_{i=-N}^NF_\Omega$ such that $\pi^{-1}_{\{-N,\ldots,N+1\}}(\tilde{U})\se U$ and $\pi^{-1}_{\{-N,\ldots,N+1\}}(\tilde{V})\se V$.
	
	By Lemma~\ref{Lemma: Open sets}, there are sets $W_1,W_2$ that are open in $F_\Lambda$ with 
	\begin{align*}
		W_1&\se\pi_{\{N+1,N+2\}}
		\left(
		\tilde{U}\star F_\Lambda
		\right)\\
		W_2&\se\pi_{\{-(N+1),-N\}}
		\left(
		F_\Lambda\star\tilde{V}
		\right).
	\end{align*}
	Since$(I_{F_\Lambda},\sigma_{F_\Lambda})$ s topologically transitive, there exists $\mathbf{x}\in I_{F_\Lambda}$ and an integer $k\geq 2$ such that $(x_{N+1},x_{N+2})\in W_1$ and $(x_{N+1+k,N+2+k})\in W_2$.
	
	Since 
	\[
	W_1\se\pi_{\{N+1,N+2\}}
	\left(
	\tilde{U}\star F_\Lambda
	\right),
	\]
	there exists $(y_{-N},y_{-(N-1)},\ldots,y_N,y_{N+1},y_{N+2})\in\tilde{U}\star F_{\Lambda}$ such that $(y_{N+1},y_{N+2})=(x_{N+1},x_{N+2})$, and since
	\[
	W_2\se\pi_{\{-(N+1),-N\}}
	\left(
	F_\Lambda\star\tilde{V}
	\right),
	\]
	there exists $(z_{-(N+1)},z_{-N},\ldots,z_{N+1})\in F_\Lambda\star\tilde{V}$ such that $(z_{-(N+1)},z_{-N})=(x_{N+1+n,N+2+n})$. Then we have that
	\[
	\left(y_{-N},\ldots,y_N,x_{N+1},x_{N+2},\ldots,x_{N+1+k},x_{N+2+k},z_{-(N-1)},z_{-(N-2)},\ldots,z_{N}\right)\in\tilde{U}\star\left(\overset{N+1+n}{\underset{i=N+1}{\bigstar}}F_\Lambda\right)\star\tilde{V}.
	\]
	
	In particular, there exists $\mathbf{w}\in U$ such that $\sigma_{F_\Omega}^{2N+1+k}(\mathbf{w})\in V$. Therefore $(I_{F_\Omega},\sigma_{F_\Omega})$ is topologically transitive. 
	
	The proof of \eqref{LF-transitive/mixing 2} is nearly identical. The one difference is that if $(I_{F_\Lambda},\sigma_{\Lambda})$ is topologically mixing, then for any sufficiently large $k$ we may choose $\mathbf{x}\in I_{F_\Lambda}$ so that
	$(x_{N+1},x_{N+2})\in W_1$ and $(x_{N+1+k}x_{N+2+k})\in W_2$. With that minor adjustment, the rest of the proof remains the same to show  that if $(I_{F_\Lambda},\sigma_{F_\Lambda})$ is topologically mixing, then $(I_{F_\Omega},\sigma_{F_\Omega})$ is topologically mixing.
\end{proof}

There is a topological semi-conjugacy from $(I_{F_\Omega},\sigma_{F_\Omega})$ to $(I^+_{F_\Omega},\sigma^+_{F_\Omega})$, so if we conclude that the shift on the two-sided Mahavier product is topologically transitive (mixing), then it follows that the shift on the one-sided Mahavier product is topologically transitive (mixing) as well.

Additionally, it is shown in \cite[Theorem 4.3]{banic2} that for any pair of positive real numbers $(\omega_1,\omega_2)$ that never connect, the Mahavier dynamical system $(I_{F_{\{\omega_1,\omega_2\}}},\sigma_{F_{\{\omega_1,\omega_2\}}})$ is topologically transitive. From Definition~\ref{Definition: LF-inducing}, for any LF-inducing set $\Omega$, the pair $(\omega_1,\omega_2) \in \mathcal{NC}$. The following corollary then follows.

\begin{corollary}\label{Corollary: LF-transitive}
	Let $\Omega$ be an LF-inducing set. Then $(I_{F_\Omega},\sigma_{F_\Omega})$ and $(I^+_{F_\Omega},\sigma^+_{F_\Omega})$ are topologically transitive.
\end{corollary}

\begin{prop}\label{Proposition: LF-inducing mixing}
	Let $\Omega=\left\{\omega_1,\ldots,\omega_n\right\}$ be an LF-inducing set. For each positive integer $M$, define the set
	\[
	\mathcal{A}(M)=\left\{\prod_{i=1}^n\omega_i^{m_i}\colon \forall i\in\{1,\ldots,n\},~m_i\geq0,\text{ and }\sum_{i=1}^nm_i=M\right\}.
	\]
	If there exists a positive integer $K$ such that for every integer $M\geq K$, $1\in\mathcal{A}(M)$, then $(I_{F_\Omega},\sigma_{F_\Omega})$ is topologically mixing.
\end{prop}

\begin{proof}
	If $1\in\mathcal{A}(M)$, then $\Delta_I\se F_\Omega^M$, so the result follows from Proposition~\ref{Proposition: Eventual diagonal and transitive implies mixing}.
\end{proof}

\begin{prop}\label{Proposition:Contains diagonal means no shadowing}
	Let $\Omega$ be an LF-inducing set. If there exists a positive integer $M$ such that 1 is an element of the set
	\[
	\mathcal{A}(M)=\left\{\prod_{i=1}^n\omega_i^{m_i}\colon \forall i\in\{1,\ldots,n\},~m_i\geq0,\text{ and }\sum_{i=1}^nm_i=M\right\},
	\]
	then $(I_{F_\Omega},\sigma_{F_\Omega})$ and $(I^+_{F_\Omega},\sigma^+_{F_\Omega})$ do not have shadowing.
\end{prop}
\begin{proof}
	We show that $(I_{F_\Omega},\sigma_{F_\Omega})$ does not have the shadowing property. The proof for $(I^+_{F_\Omega},\sigma^+_{F_\Omega})$ is essentially identical.
	
	Suppose $M\in\N$ and $1\in\mathcal{A}(M)$. Then we may choose a finite sequence $(\omega_{j_1},\omega_{j_2},\ldots,\omega_{j_M})$ such that 
	$
	\prod_{i=1}^M\omega_{j_i}=1.
	$
	Without loss of generality, we may suppose $\omega_{j_1}\leq\omega_{j_2}\leq\cdots\leq\omega_{j_M}$. This implies that for all $n\in\{1,\ldots,M\}$, 
	\begin{equation}
		\prod_{i=1}^n\omega_{j_i}\leq 1.\label{Eqn: omega products bounded by 1}
	\end{equation}

	Fix $a\in (0,1)$ and $\rho>0$ such that if $(x,y)\in F_\Omega^M$ and $\max\{|x-a|,|y-a|\}<\rho$, then $y=x$. Choose $\e>0$ with $\e<\min\{\rho/2^M,(1-a)/2\}$. Let $\delta>0$. There is a finite, increasing sequence $(z_0,z_2,\ldots,z_k)$ such that $z_0=a$, $z_k=1$, and for all $i\in\{0,\ldots,k-1\}$, $|z_{i+1}-z_i|=z_{i+1}-z_i<\delta$.
	
	For each $l\in\{0,\ldots,k\}$ define $\hat{z}^l_0=z_l$ and for each $n\in\{1,\ldots,M\}$, define
	\[
	\hat{z}^l_n=\left(\prod_{i=1}^n\omega_{j_i}\right)z_l.
	\]
	(Note that $\hat{z}^l_M=\hat{z}^l_0=z_l$.) From Inequality~\ref{Eqn: omega products bounded by 1}, we have that for all $l\in\{0,\ldots,k\}$ and $n\in\{1,\ldots,M\}$,
	\begin{equation}
		\left|\hat{z}^l_n-\hat{z}^{l+1}_n\right|<\left|z_l-z_{l+1}\right|<\delta\label{Eqn: z-hat-inequality}
	\end{equation}
	
	For each integer $l\geq 0$, we define a point $\mathbf{x}(l)=(x^l_i)_{i\in\Z}$ in $I_{F_\Omega}$ as follows:
	\begin{itemize}
		\item If $l=tM$ for some $t\in\{0,\ldots,k\}$, we set $x^{l}_i=\hat{z}^t_{i\text{ (mod }M\text{)}}$ for all $i\in\Z$.
		\item Given $t\in\{0,\ldots,k-1\}$, if $l\in\{tM+1,\ldots,(t+1)M-1\}$, we set $\mathbf{x}(l)=\sigma(\mathbf{x}(l-1))$.
		\item For $l>k$, we set $\mathbf{x}(l)=\sigma\left(\mathbf{x}(l-1)\right)$.
	\end{itemize} 
	We claim that $(\mathbf{x}(l))_{l=0}^\infty$ is a $\delta$-pseudo-orbit for $(I_{F_\Omega},\sigma_{F_\Omega})$. For any $l\notin\{M,2M,\ldots,kM\}$, we have $\sigma(\mathbf{x}(l-1))=\mathbf{x}(l)$, so we need only to check that $D(\sigma(\mathbf{x}(l)),\mathbf{x}(l+1))<\delta$ if $l=tM-1$ for some $t\in\{1,\ldots,k\}$.
	
	Let $t\in\{1,\ldots,k\}$, and let $l=tM-1$. Then we get
	\begin{align*}
		D\left[\sigma\left(\mathbf{x}(l)\right),\mathbf{x}(l+1)\right]&=D\left[\sigma\left(\mathbf{x}(tM-1)\right),\mathbf{x}(tM)\right]\\
		&=D\left[\sigma^M\left(\mathbf{x}((t-1)M)\right),\mathbf{x}(tM)\right]\\
		&=D\left[\mathbf{x}((t-1)M),\mathbf{x}(tM)\right]\\
		&=\max\left\{\frac{\left|\hat{z}^{t-1}_{i}-\hat{z}^{t}_{i}\right|}{2^{|i|}}\colon i\in\{0,\ldots,M-1\}\right\}\\
		&<\delta,
	\end{align*}
	with the last inequality following from \ref{Eqn: z-hat-inequality}. Therefore $(\mathbf{x}(l))_{l=0}^\infty$ is a $\delta$-pseudo-orbit.
	
	Now we show that  it cannot be $\e$-shadowed by a true orbit. Suppose $\mathbf{y}\in I_F$ satisfies $D(\sigma^l(\mathbf{y}),\mathbf{x}(l))<\e$ for all integers $l\geq 0$. Then for all $t\in\{0,\ldots,k\}$, we have
	\begin{align*}
		\frac{\rho}{2^M}>\e>D\left[\sigma^{tM}(\mathbf{y}),\mathbf{x}(tM)\right]&\geq\max\left\{\left|y_{tM}-x^{tM}_0\right|,\frac{\left|y_{tM+M}-x^{tM}_M\right|}{2^M}\right\}\\
		&=\max\left\{\left|y_{tM}-z_t\right|,\frac{\left|y_{tM+M}-z_t\right|}{2^M}\right\}.
	\end{align*}
	Thus the point $(y_{tM},y_{tM+M})\in F_\Omega^M$ satisfies $\max\{|y_{tM}-z_t|,|y_{tM+M}-z_t\}<\rho$, so by our choice of $\rho$, we have $y_{tM+M}=y_{tM}$. 
	
	Therefore the sequence $(y_0,y_M,y_{2M},\ldots,y_{kM})$ is a constant sequence. However, we have $|y_0-z_0|=|y_0-a|<\e$, and $|y_{kM}-z_M|=|y_{kM}-1|<\e$. This contradicts the choice of $\e<(1-a)/2$. Therefore $(I_F,\sigma_F)$ does not have the shadowing property.
\end{proof}

Note that if an LF-inducing set $\Omega$ is closed under multiplicative inverses, then $F_\Omega^{-1}=F_\Omega$, and we may apply Theorem~\ref{Theorem: F union F^{-1} has dense periodic}. From this observation and the rest of the results on the dynamics of LF-inducing relations, we get the following examples.

\begin{example}~
	None of the examples discussed in this section have the shadowing property (Proposition~\ref{Proposition:Contains diagonal means no shadowing}) but they have the following.
	\begin{enumerate}
		\item The shift map $\sigma_{F_{\{1/2,3,1/3,2\}}}$ is a homeomorphism on the Lelek fan $I_{F_{\{1/2,3,1/3,2\}}}$ (Proposition~\ref{Proposition: LF-inducing yields Lelek fan (two-sided)}) such that $(I_{F_{\{1/2,3,1/3,2\}}},\sigma_{F_{\{1/2,3,1/3,2\}}})$ is Devaney chaotic (Corollary~\ref{Corollary: LF-transitive} and Theorem~\ref{Theorem: F union F^{-1} has dense periodic}).
		\item The shift map $\sigma^+_{F_{\{1/2,3,1/3,2\}}}$ is a non-injective, continuous function on the Lelek fan $I^+_{F_{\{1/2,3,1/3,2\}}}$ (Proposition~\ref{Proposition: LF-inducing yields Lelek fan (one-sided)}) such that $(I^+_{F_{\{1/2,3,1/3,2\}}},\sigma^+_{F_{\{1/2,3,1/3,2\}}})$ is Devaney chaotic (Corollary~\ref{Corollary: LF-transitive} and Theorem~\ref{Theorem: F union F^{-1} has dense periodic}).
		\item The shift map $\sigma_{F_{\{1/2,3,1/3,2,1\}}}$ is a homeomorphism on the Lelek fan $I_{F_{\{1/2,3,1/3,2,1\}}}$ (Proposition~\ref{Proposition: LF-inducing yields Lelek fan (two-sided)}) such that $(I_{F_{\{1/2,3,1/3,2,1\}}},\sigma_{F_{\{1/2,3,1/3,2,1\}}})$ is Devaney chaotic (Corollary~\ref{Corollary: LF-transitive} and Theorem~\ref{Theorem: F union F^{-1} has dense periodic}) and topologically mixing (Proposition~\ref{Proposition: LF-inducing mixing}), and has the specification property (Theorem~\ref{Theorem: three line has specification}).
		\item The shift map $\sigma^+_{F_{\{1/2,3,1/3,2,1\}}}$ is a non-injective, continuous function on the Lelek fan $I^+_{F_{\{1/2,3,1/3,2,1\}}}$ (Proposition~\ref{Proposition: LF-inducing yields Lelek fan (one-sided)}) such that $(I^+_{F_{\{1/2,3,1/3,2,1\}}},\sigma^+_{F_{\{1/2,3,1/3,2,1\}}})$ is Devaney chaotic (Corollary~\ref{Corollary: LF-transitive} and Theorem~\ref{Theorem: F union F^{-1} has dense periodic}) and topologically mixing (Proposition~\ref{Proposition: LF-inducing mixing}), and has the specification property (Theorem~\ref{Theorem: three line has specification}).
	\end{enumerate}
\end{example}

\section{Open problems}

We conclude the paper by stating the following open problems.

\begin{problem}
	Does Lelek fan admit a surjective shadowing map?
\end{problem}

\begin{problem}
Is the Lelek fan the only smooth fan that admits homeomorphisms with the specification property?
\end{problem}

\section{Acknowledgments}

This work is partially funded by European Union - NextGenerationEU grant IP-UNIST-44.

\vphantom{}

\noindent  Goran Erceg\\
             Faculty of Science, University of Split, Rudera Bo\v skovi\' ca 33, Split 21000, Croatia\\
{{goran.erceg@pmfst.hr}       }  \\

\noindent  James  Kelly\\
             Department of Mathematics,  Christopher Newport University, 1 Ave. of the Arts, Newport News, Virginia 23606, USA\\
{{james.kelly@cnu.edu}       }  \\

\noindent  Judy  Kennedy\\
             Department of Mathematics,  Lamar University, 200 Lucas Building, P.O. Box 10047, Beaumont, Texas 77710, USA\\
{{kennedy9905@gmail.com}       }  \\

\noindent Christopher Mouron \\
              Department of Mathematics and Statistics, Rhodes College, 2000 North Parkway, Memphis, TN 38112, USA \\
                  {mouronc@rhodes.edu}           

\vphantom{}
     
\noindent Van Nall\\
             Department of Mathematics, University of Richmond, 
               410 Westhampton Way, University of Richmond, VA 23172, USA\\
               {vnall@richmond.edu}       				\-
				

\end{document}